\documentclass[11pt,twoside]{article} 

\usepackage{eqnarray,amsmath}
\usepackage[utf8]{inputenc} 
\usepackage[T1]{fontenc}    
\usepackage{booktabs}       
\usepackage{amsfonts}       
\usepackage{nicefrac}       
\usepackage{microtype}      
\usepackage{xcolor}         
\usepackage{subcaption}
\expandafter\def\csname 
ver@subfig.sty\endcsname{}
\usepackage{eqnarray,amsmath}
\usepackage{epsf}
\usepackage{epsfig}
\usepackage{fancyhdr}
\usepackage{graphics}
\usepackage{graphicx}
\usepackage{psfrag}
\usepackage{fullpage}
\usepackage{pdfpages}
\usepackage{amssymb}
\usepackage{natbib}

\usepackage{amsthm}
\usepackage{comment}

\usepackage{url}

\usepackage{color}

\usepackage{amsthm}
\usepackage{amsmath}
\usepackage{amssymb,bbm}
\usepackage[skip=0pt]{caption}  

\usepackage{textcomp}
\usepackage{siunitx}
\usepackage{wrapfig}
\usepackage{algorithm}

\usepackage{multirow}
\usepackage{multicol}
\usepackage{makecell}
\usepackage{colortbl} 
\usepackage{changepage}

\definecolor{customyellow}{HTML}{fedf8a}
\usepackage{eqnarray,amsmath}
\usepackage{epsf}
\usepackage{epsfig}
\usepackage{fancyhdr}
\usepackage{graphics}
\usepackage{graphicx}
\usepackage{psfrag}
\usepackage{fullpage}
\usepackage{pdfpages}
\usepackage{ragged2e}

\newtheorem*{remark}{Remark}

\newtheorem{assumption}{Assumption}

\newtheorem{lemma}{Lemma}

\newtheorem{theorem}{Theorem}
\newtheorem{proposition}{Proposition}
\newtheorem{definition}{Definition}

\usepackage{eqnarray,amsmath}
\usepackage{booktabs}       
\usepackage{nicefrac}       

\usepackage{subfig}
\usepackage{epsf}
\usepackage{epsfig}
\usepackage{fancyhdr}
\usepackage{graphics}
\usepackage{graphicx}
\usepackage{epstopdf}
\usepackage{psfrag}
\usepackage{fullpage}
\usepackage{pdfpages}

\usepackage{enumerate}   
\usepackage{multirow}
\usepackage{bm}

\usepackage{mathtools}

\usepackage{url}
\usepackage[colorlinks,linkcolor=black,citecolor=blue, pagebackref=true]{hyperref}
\renewcommand*{\backrefalt}[4]{%
    \ifcase #1 \footnotesize{(Not cited.)}%
    \or        \footnotesize{(Cited on page~#2.)}%
    \else      \footnotesize{(Cited on pages~#2.)}%
    \fi}

\usepackage{color}

\usepackage{amsthm}
\usepackage{amsmath}
\usepackage{amssymb,bbm}
\usepackage{caption}
\usepackage{algorithmic}
\usepackage{algorithm}
\usepackage{textcomp}
\usepackage{siunitx}
\usepackage{wrapfig}
\usepackage{algorithmic}
\usepackage{algorithm}
\usepackage{multirow}
\usepackage{multicol}
\usepackage{mathtools}

\def\1{\bm{1}}

\DeclareMathAlphabet{\mathsfit}{\encodingdefault}{\sfdefault}{m}{sl}
\SetMathAlphabet{\mathsfit}{bold}{\encodingdefault}{\sfdefault}{bx}{n}

\usepackage{eqnarray,amsmath}

\usepackage{amsthm}
\usepackage{amsmath}
\usepackage{amssymb,bbm}

\allowdisplaybreaks

\begin{document}

\begin{center}

{\bf{\LARGE{Partial Differential Equation Barriers to Identifiability in Infinite Mixture Models
}}}
  
\vspace*{.2in}
{\large{
\begin{tabular}{cccc}
Dung Le\footnotemark & Nicola Bariletto\footnotemark[1] & Alessandro Rinaldo & Nhat Ho
\end{tabular}
}}

\footnotetext{Equal contribution.}

\vspace*{.2in}

\begin{tabular}{c}
The University of Texas at Austin
\end{tabular}

\vspace*{.2in}
\today


\begin{abstract}
We study identifiability of mixing measures in infinite mixture
models. We show that, in many common cases, lack of identifiability can be characterized in terms of certain differential structures of the kernel family with respect to its parameters. In our main results, we prove that when the
kernel is annihilated by a non-trivial differential or difference-differential operator over the
parameter space, there exist infinitely many distinct mixing measures yielding the same mixture density.
We give verifiable conditions for such operators to exist, covering many common cases, including the location-scale Gaussian, location-scale Student-$t$, Gamma, Beta, Dirichlet,
negative binomial and non-central Chi-squared families. Furthermore, our conditions apply to any exponential family whose parameter dimension exceeds the
dimension of its sufficient statistic and, more generally, to kernels with polynomially-growing score functions. We complement our results with a minimax lower bound on the estimation error for the mixing measure in the Wasserstein distance under non-identifiability. On the flips side, we describe three classes of kernels for which identifiability is preserved and nonparametric statistical inference remains possible.
\end{abstract}
\end{center}





\section{Introduction}\label{sec:introduction}

Mixtures are among the most widely used models in statistics, offering a
principled way to describe populations composed of unobserved subgroups and to
build flexible distributions from simple parametric components
\citep{mclachlan2000finite,fraley2002model}. In their basic form, data
$X_1,\ldots,X_n$ taking values in a sample space $\mathcal X\subseteq\mathbb
R^{d_1}$ are assumed to be independent draws from a density
\begin{equation}\label{eq:model}
    p_{G_*}(x)=\int_{\Theta}f(x\mid\theta)\,G_*(d\theta), \qquad x \in \mathcal{X},
\end{equation}
where $\{f(\cdot\mid\theta):\theta\in\Theta\}$ is a family of probability
density kernels (with respect to some common dominating measure on $\mathcal{X}$, say $\lambda$) indexed by the parameter space $\Theta\subseteq\mathbb
R^{d_2}$, and the \emph{mixing measure} $G_*$ is a probability measure on
$\Theta$ that weights the parameter values entering the mixture. Equivalently,
the generative model implied by~\eqref{eq:model} admits a hierarchical representation: latent parameters
$\theta_1,\ldots,\theta_n$ are drawn independently from $G_*$, and each $X_i$ is
then drawn independently from $f(\cdot\mid\theta_i)$. For the purposes of
statistical inference, two related objects are of interest: the density
$p_{G_*}$ and the mixing measure $G_*$. The latter, whose estimation motivates
the present work, encodes the latent structure underlying
tasks such as model-based clustering and the recovery of subgroup parameters.

In nonparametric settings, the unknown $G_*$ is left unrestricted (or nearly so) and is modelled
through mixing measures $G$ with infinite support, which we call \emph{infinite
mixing measures}, giving rise to \emph{infinite mixture models}. This flexible
formulation is standard in Bayesian nonparametrics, through priors on $G$ such
as the Dirichlet process \citep{Ferguson,ferguson1983bayesian,lo1984class}
that model $G_*$ as discrete and induce a latent clustering of the data
according to the mixture component imputed to have generated each observation.
While $G_*$ governs this rich latent structure, the flexibility of the scheme
comes at the cost that $G_*$ is considerably harder to estimate than $p_{G_*}$. This is because, as the
hierarchical form makes explicit, the data are directly informative about $p_{G_*}$, whereas information about $G_*$ is only indirectly available through the additional noise layer
contributed by the kernel $f(\cdot\mid\theta)$. This effectively makes $p_{G_*}$ a noisy
version of $G_*$, so that the statistical recovery of $G_*$ may be framed as the problem of denoising $p_{G_*}$, or an estimate thereof, to extract the latent signal $G_*$. In light of this interpretation, the first fundamental question to address, for the task at hand to be at least well-posed, is about \emph{identifiability}: even granting perfect estimation of the data-generating density---that is, $p_{\hat G}=p_{G_*}$ for some estimator $\hat G$ of $G_*$---does recovery of
the mixing measure follow---that is, $\hat G=G_*$?

\subsection{Contributions of the article}\label{sec:contributions}

This work focuses precisely on the above-mentioned question of identifiability for infinite mixture models. To begin our discussion, write $\mathcal A$ for the linear operator $G\mapsto\int_\Theta f(\cdot\mid\theta)G(d\theta)$ that maps mixing measures to mixture densities. Using this formalism, the mixing measure is identifiable based on the mixture density whenever $\mathcal A$ is injective. In light of this, the key questions we address are: for which kernels $f$ does injectivity of $\mathcal A$ hold? And when this property does not hold, what are the structural mechanisms responsible for the failure of identifiability?

Our starting point is the observation that many standard kernels belong to the null space of, or equivalently are \emph{annihilated} by, a differential operator acting on the parameter $\theta$. To build intuition, fix a probability measure $G_*$ admitting a density $g_*$ that is bounded away from zero on an open ball $B\subseteq \Theta^\circ$, where $\Theta^\circ$ denotes the interior of $\Theta$, and suppose that there exists a non-trivial linear differential operator $\mathcal L_\theta$ such that
\begin{equation}\label{eq:pde_barrier}
    \mathcal L_\theta f(x\mid\theta)=0\qquad\text{for } \lambda\text{-almost every  }x\in\mathcal X\text{ and every }\theta\in\Theta^\circ.
\end{equation}
We refer to \eqref{eq:pde_barrier} as a \emph{parameter partial differential equation}, or \emph{parameter PDE}. For ease of exposition, we postpone the presentation of more general differential structures and of concrete examples. For the moment, let $\mathcal L_\theta^*$ denote the formal adjoint of $\mathcal L_\theta$ \citep{folland1999real}, the latter also being a differential operator that can be applied to any infinitely differentiable and compactly supported function $u:\mathbb R^{d_2}\to\mathbb R$. Hence, letting $g:=\mathcal L^*_\theta u$ for an appropriately chosen $u$, the properties of adjoint operators and the parameter PDE lead to
\begin{equation}\label{eq:duality}
    \int_\Theta f(x\mid\theta)g(\theta)\,d\theta=\int_\Theta u(\theta)\mathcal L_\theta f(x\mid\theta)\,d\theta=0\qquad \text{for }\lambda \text{-almost every } x\in\mathcal X.
\end{equation}
Integrating \eqref{eq:duality} over $x$ and using Fubini's theorem also gives $\int_\Theta g(\theta)\,d\theta=0$. Therefore, defining $G(d\theta):=g(\theta)\,d\theta$ yields a signed measure such that $\mathcal AG=0$. This, together with other properties of $u$, implies that the perturbed probability measure $G_*+\varepsilon G$, for $\varepsilon>0$ small enough, is different from $G_*$ while giving rise to the same mixture density: $p_{G_*}= p_{G_*+\varepsilon G}$ $\lambda$-almost everywhere.

This simple line of reasoning sheds new light on a direct connection between the presence of parameter PDEs and the non-identifiability of mixing measures in infinite mixture models. Importantly, the existence of parameter PDEs as in \eqref{eq:pde_barrier} is not a mere mathematical artifact, but will be shown to rather be a pervasive feature of commonly used kernel families. While a comprehensive set of such families will be discussed later, it is useful to exemplify the discussion by means of a concrete and popular case, namely, \emph{location-scale Gaussian mixtures}. With $d_1=1$ and $d_2=2$, let $\mathcal N(x\mid\mu,\nu):=(2\pi\nu)^{-1/2}\exp\{-(x-\mu)^2/(2\nu)\}$ denote the univariate Gaussian density with mean $\mu$ and variance $\nu$, with kernel $f(x\mid\mu,\nu)=\mathcal N(x\mid\mu,\nu)$ indexed by $\theta:=(\mu,\nu)\in\Theta:=[-a,a]\times[\nu_{\min},\nu_{\max}] \subset \mathbb{R}^2$, for $a>0$ and $0<\nu_{\min}<\nu_{\max}<\infty$. Denoting by $\partial_y$ the partial derivative operator with respect to the scalar argument $y$, the Gaussian kernel famously solves the heat equation
\begin{equation}
\label{eq:heat}
    \partial_\nu f(x\mid\mu,\nu)=\tfrac{1}{2}\partial^2_\mu f(x\mid\mu,\nu)\qquad \text{for every }x\in\mathbb R \text{ and } (\mu,\nu)\in\Theta^\circ,
\end{equation}
which is constitutes a parameter PDE in the sense of \eqref{eq:pde_barrier}, with $\mathcal L_\theta=\partial_\nu-\frac12\partial^2_\mu$. Hence, fixing an appropriate infinitely differentiable $u:\mathbb R^2\to\mathbb R$ and recognizing that $\mathcal L_\theta^*=-\partial_\nu u(\theta,\nu)-\tfrac12\partial_\theta^2u(\theta,\nu)$, the same line of reasoning as in the previous paragraph implies that infinite location-scale Gaussian mixtures suffers from non-identifiability of the mixing measure.

Beyond these heuristic preliminaries, the contributions of the article can be summarized as follows. First, we formally show that the arguments outlined above hold in very general settings. In particular, Theorem~\ref{thm:homogeneous} states that a kernel annihilated by a non-trivial differential operator generates a non-identifiable mixture model whenever the mixing measure has a density that is locally bounded away from zero. Theorem~\ref{thm:shift} extends this conclusion to difference-differential operators, in which the kernel is evaluated at shifted parameter values. Second, we give verifiable conditions under which such operators exist. Proposition~\ref{prop:exp_family_pde} shows that any exponential family whose parameter dimension exceeds the dimension of its sufficient statistic admits a parameter PDE, while Proposition~\ref{prop:general_pde} generalizes this results also to kernels whose score is polynomial in a finite collection of statistics. Together with the examples of Section~\ref{sec:pde_examples}, these results identify common barriers to identifiability, which fails across the location-scale Gaussian, location-scale Student-$t$, gamma, beta, Dirichlet, negative binomial and non-central chi-squared families, and they indicate that over-parameterisation, in the sense $d_2>q$ with $q$ denoting the dimension of the sufficient statistic, is a systematic mechanism behind such barriers. As we document in Proposition~\ref{pro:statistical_estimation}, these instances of non-identifiability have crucial implications for statistical estimation of the mixing measure, which becomes impossible without further constraints on the model or $G_*$.

Third, we describe three classes of kernels for which no such barrier exists and identifiability is ensured. Proposition~\ref{prop:translation} introduces generalized translation families, for which identifiability reduces to the non-vanishing of the Fourier transform of a base kernel. Proposition~\ref{prop:exp_family_identif} treats exponential families in which the sufficient statistic has dimension at least that of the parameter and its range is rich enough to determine measures on the natural parameter space, a condition of Müntz--Szász type. Proposition~\ref{prop:super_exponential} considers a third class of kernels that both preserves identifiability and verifies the absence of any parameter PDE condition.

The remainder of the article is organised as follows. Section~\ref{sec:notation} fixes notation and preliminary concepts. Section~\ref{sec:literature} reviews the existing literature on identifiability and estimation in mixture models. Section~\ref{sec:pde_barriers} develops the theory of PDE barriers to identifiability, presents concrete examples, and discusses the consequences of non-identifiability on statistical estimation of the mixing measure. Section~\ref{sec:identifiable} treats kernel families for which identifiability is preserved. Section~\ref{sec:discussion} concludes. The proofs of the results presented in the article are collected in the Supplementary Material.

\section{Notation and preliminary concepts}\label{sec:notation}

We assume throughout that the sample space $\mathcal X$ is a Borel-measurable subset of $\mathbb R^{d_1}$ over which a $\sigma$-finite Borel measure $\lambda$ is well defined; this is understood to be the Lebesgue measure for continuous kernels and the counting measure for discrete ones. Analogously, the parameter space $\Theta\subset\mathbb R^{d_2}$ is assumed to be regular compact \footnote{A subset $S$ of Euclidean space $\mathbb{R}^d$ is call \textit{regular closed set} if $S$ equals to the closure of its interior (\cite{willard1970general}). } with a non-empty interior $\Theta^\circ$, and is equipped with the restriction of  the Lebesgue measure in $\mathbb R^{d_2}$. For $k\in\mathbb N$, $k$-fold product spaces are denoted as $\mathcal X^k$ and $\Theta^k$, and are endowed (when meaningful) with the product topology and the corresponding Borel $\sigma$-algebra. Similarly, $P^k$, where $P$ is a probability measure on some measurable space such as $\mathcal X$ or $\Theta$, denotes the $k$-fold product measure on the corresponding product space. A kernel $f:\mathcal X\times\Theta\to[0,\infty)$ is taken to be a jointly measurable function satisfying $\int_{\mathcal X}f(x\mid\theta)\,\lambda(dx)=1$ for every $\theta\in\Theta$, so that $x\mapsto f(x\mid \theta)$ is a probability density function with respect to $\lambda$.

We write $\mathcal M(\Theta)$ for the Banach space of finite signed Radon measures on $\Theta$ endowed with the total variation norm $\|\cdot\|_{\mathrm{TV}}$,  $\mathcal P(\Theta)\subset\mathcal M(\Theta)$ for the set of Borel probability measures on $\Theta$, and $\mathcal D(\Theta)\subset \mathcal P(\Theta)$ for the set of discrete probability measures (i.e., those with at most countable support). The mixture operator $\mathcal A:\mathcal M(\Theta)\to L^1(\mathcal X,\lambda)$ is defined by
\begin{equation*}
 G \mapsto (\mathcal AG)(x):=\int_\Theta f(x\mid\theta)\,G(d\theta),\qquad x\in\mathcal X,
\end{equation*}
and we write $p_G:=\mathcal AG$ whenever $G\in\mathcal P(\Theta)$, so that $p_G$ is a probability density with respect to $\lambda$.

The mixture model \eqref{eq:model} is said \emph{identifiable} if, for all $G_1,G_2\in\mathcal P(\Theta)$, the equality $p_{G_1}=p_{G_2}$, holding $\lambda$-almost everywhere, implies $G_1=G_2$. By linearity of the mixture operator,  identifiability is implied by the injectivity of $\mathcal A$ on $\mathcal M(\Theta)$; that is, the model is identifiable if
\begin{equation*}
    \mathcal AG=0 \;\lambda\text{-almost everywhere} \quad\implies\quad G=0
\end{equation*}
for all $G\in\mathcal M(\Theta)$, where $G = 0$ indicates that $G$ is the identically zero signed measure . Also, following the seminal contribution of \cite{Nguyen-13} we will, when discussing statistical estimation, measure the discrepancy between any two mixing measures $G_1,G_2\in\mathcal P(\Theta)$ according to the \emph{Wasserstein distance of order 1},
\begin{equation*}
    W_1(G_1,G_2):=\inf_{\pi\in\Pi(G_1,G_2)} \int_\Theta \|\theta_1-\theta_2\|\,\pi(d\theta_1,d\theta_2),
\end{equation*}
where $\Pi(G_1,G_2)$ is the set of probability measures $\pi$ on $\Theta^2$ with left- and right-marginal $G_1$ and $G_2$, respectively \citep{Villani-09}.


For a multi-index $\alpha=(\alpha_1,\ldots,\alpha_{r})\in\mathbb N_0^{r}$, where $\mathbb N_0:=\{0\}\cup\mathbb N$, we set $|\alpha|:=\alpha_1+\cdots+\alpha_{r}$ and write $\partial^\alpha_\theta:=\partial^{|\alpha|}/\partial\theta_1^{\alpha_1}\cdots\partial\theta_{r}^{\alpha_{r}}$, with the convention that $\partial^\alpha_\theta$ is the identity when $|\alpha|=0$. To ease notation, we will often omit $\alpha \in \mathbb N^{r}$ from subscripts and set definitions, and will instead directly state the required conditions on $\alpha$ (e.g., in terms of its size $|\alpha|$). For a non-empty open set $U\subseteq\mathbb R^{d_2}$ and $k\in\mathbb N_0\cup\{\infty\}$, $C^k(U)$ denotes the space of $k$ times continuously differentiable real functions on $U$, and $C_c^\infty(U)$ the space of infinitely differentiable functions whose support is a compact subset of $U$; elements of $C_c^\infty(U)$ are referred to as test functions. For real functions $w_1,w_2$ on $\Theta$ with $w_1w_2$ integrable, we write $\langle w_1,w_2\rangle:=\int_\Theta w_1(\theta)w_2(\theta)\,d\theta$ for the $L^2(\mathcal X)$ inner product with respect to the Lebesgue measure on $\Theta$. We denote by $B(\theta_0,\rho)$ the open Euclidean ball of centre $\theta_0$ and radius $\rho$, by $\bar A$ the closure of a set $A$, and by $e_1,\ldots,e_{d}$ the standard basis of $\mathbb R^{d}$. For $h\in L^1(\mathbb R^{d})$, the Fourier transform of $h$ is
\begin{equation*}
    \mathcal F[h](\xi):=\int_{\mathbb R^{d}}e^{-i\xi^\top x}h(x)\,dx,\qquad \xi\in\mathbb R^{d}.
\end{equation*}
Analogously, for a finite measure $\mu$ on $\mathbb{R}^d$, its characteristic function is denoted with
\[
 \mathcal F[\mu](\xi):=\int_{\mathbb R^{d}}e^{ i\xi^\top x}d\mu(x),\qquad \xi\in\mathbb R^{d}.
\]
Finally, we call a map between metric spaces a homeomorphism onto its image if it is a continuous bijection onto its image with continuous inverse. We refer to \cite{folland1999real} for further background on the previous and other functional-analytic tools employed throughout the article.
\section{Relation to existing work}\label{sec:literature}
Identifiability and estimation in mixture models have a long history, going back
to the work of
\cite{teicher1960mixture,teicher1961identifiability,teicher1963finite}. Much of
the recent literature has dealt with estimation of the density $p_{G_*}$,
with \cite{Ghosal-1999,Barron-Shervish-Wasserman-99,Ghosal-2000,Ghosal-2001}
pioneering the study of $L^1$ consistency and contraction rates in the
infinite-dimensional setting. The problem of quantifying convergence at the
level of the mixing measure has been recently formalised by \cite{Nguyen-13}, who
advocated for the Wasserstein distances as natural tools toward this aim on
account of their geometrically meaningful behaviour between measures with
arbitrary supports. This perspective has generated a growing body of work on estimation of the mixing measure
\citep{Ho-Nguyen-Ann-16,Ho-Nguyen-EJS-16,Gao2016Posterior,scricciolo2018bayes,rousseau2024wasserstein,bariletto2026convergence,ascolani2026posterior}.
A central device in these analyses is the search for inequalities of the form
$\int_\mathcal X |p_{G_1}(x)-p_{G_2}(x)|\,\lambda(dx)\ge \varphi(W_r(G_1,G_2))$,
for some increasing $\varphi$ and a Wasserstein distance $W_r$ of order $r\geq 1$ \citep[see][for a definition when $r>1$]{Villani-09}, which convert density rates into mixing-measure rates
\citep{nguyen2026optimal}. Such an inequality is possible in general only when
$\mathcal A$ is injective, so that establishing it may be viewed as a
quantitative refinement of identifiability itself; the attainable $\varphi$, and hence the
estimation rates, have in turn been shown to be governed by analytic features of
the kernel, notably the decay of its characteristic function and the smoothness
of the parameterisation.

Work addressing identifiability directly, rather than through estimation rates,
has largely been confined to finite mixing measures. In this setting,
\cite{Ho-Nguyen-Ann-16,Ho-Nguyen-EJS-16,manole2022refined} showed that the
identifiability and convergence behaviour of the mixing measure is controlled by some algebraic
relations among the parameter derivatives of the kernel: when certain families
of such derivatives of $f$ are linearly dependent, distinct configurations
of atoms become difficult or impossible to separate, degrading identifiability
and estimation rates; see also \cite{nguyen2026optimal} for a recent review and
\cite{ho2022convergence,nguyen2023demystifying,do2025strong,bariletto2026bayesian} for
extensions to the covariate-dependent case. Differential structure in the
parameter has thus already been recognised as a key mechanism in the finite
mixture setting, and the aim of this article is to develop these connections in
the infinite-dimensional setting as well. It is notable that, in part also owing to applications in deconvolution
problems \citep{Caillerie-etal-11,dedecker2013minimax,dedecker2015improved}
where the noise kernel $f(\cdot\mid\theta)$ is typically well understood and need not be very flexible, most of
the literature on infinite mixing-measure estimation has focused on relatively
simple kernel structures, such as location families or, only recently,
location-scale families with a scale parameter shared across components but estimated from the data
\citep{bariletto2026convergence}. As our analysis reveals, the scarcity of results
on more general and widely used families, such as location-scale ones, is not accidental; rather, many of these flexible parameterisations give rise to differential structures that are
incompatible with an identifiable model, rendering the recovery of $G_*$ from a
density estimate $p_{\hat G}$ impossible in general.

\section{PDE barriers to mixing-measure identifiability}\label{sec:pde_barriers}

We are now in a position to formally initiate our study of identifiability of mixing measures in the context of infinite mixture models. In this section, we examine two types of PDE barriers to identifiability and provide examples that illustrate both.

\subsection{Non-identifiability via parameter PDEs}\label{sec:homogeneous}

The first type of PDE barrier to identifiability we examine is induced by the following kind of differential operator.

\begin{definition}\label{def:operator}
Let $U\subseteq\Theta^\circ$ be open and $r\in\mathbb N_0$. A \emph{parameter differential operator} on $U$ of order $r$ is a map $\mathcal L_\theta$ acting on functions $w\in C^r(U)$ by the relation
\begin{equation*}
    \mathcal L_\theta w=\sum_{|\alpha|\le r}c_\alpha(\theta)\,\partial^\alpha_\theta w, \qquad \theta\in U,
\end{equation*}
with coefficients $c_\alpha\in C^{|\alpha|}(U)$. Its formal adjoint is $\mathcal L^*_\theta w=\sum_{|\alpha|\le r}(-1)^{|\alpha|}\partial^\alpha_\theta(c_\alpha w)$, so that $\langle\mathcal L_\theta w_1,w_2\rangle=\langle w_1,\mathcal L^*_\theta w_2\rangle$ for all $w_1\in C^r(U)$ and $w_2\in C_c^\infty(U)$. The operator is \emph{non-trivial} on a set $A\subseteq U$ if $c_\alpha$ is not identically zero on $A$ for at least one $\alpha$ with $|\alpha|\ge1$.
\end{definition}

Accordingly, we say that the kernel $f$ satisfies a \emph{parameter PDE} on $U$ if $\mathcal L_\theta f(x\mid\theta)=0$ for $\lambda$-almost every $x\in\mathcal X$ and every $\theta\in U$, for some parameter differential operator $\mathcal L_\theta$ on $U$. As anticipated, the presence of such parameter PDE relations leads to non-identifiability of the mixing measure.

\begin{theorem}\label{thm:homogeneous}
Let $U\subseteq\Theta^\circ$ be open and let $f$ satisfy a parameter PDE on $U$ for a parameter differential operator $\mathcal L_\theta$ that is non-trivial on an open ball $B$ with $\bar B\subset U$. Let $G_*\in\mathcal P(\Theta)$ admit a density $g_*$ with respect to the Lebesgue measure on $\Theta$ such that $g_*\ge m$ Lebesgue-almost everywhere on $B$, for some $m>0$. Then there exists $G\in\mathcal P(\Theta)$ with $G\neq G_*$ and $p_G=p_{G_*}$ $\lambda$-almost everywhere.
\end{theorem}

The proof of Theorem~\ref{thm:homogeneous} details the construction of the probability density $p_G$ obtained by perturbing the mixing measure $g_*$ by an additive term of the form $\epsilon \overline{h}$, where $\overline{h}$ is a function on $\Theta$ such that $\int_\Theta \overline{h}(\theta) d \theta = 0$ and $\epsilon > 0$ is appropriately small. It is worth emphasizing that, since $\overline{h}$ and $\epsilon$ depends only on $B$ and $m$, this construction is generic: there is a continuum of choices of $\overline{h}$ and $\epsilon$, each leading to a mixing measure that is indistinguishable,  based only on the induced mixture density, from $G_*$. Secondly, the conclusion is local in the parameter space: only the behavior of $\mathcal L_\theta$ and of $g_*$ on the ball $B$ matters, so that a kernel satisfying a parameter PDE on an arbitrarily small open region generates non-identifiability for every mixing measure whose density is bounded away from zero there. In particular, if $g_*$ is continuous and bounded away from $0$ on $\Theta^\circ$, the assumptions of the Theorem reduce to the non-triviality of $\mathcal L_\theta$.


Theorem~\ref{thm:homogeneous} links non-identifiability to the existence of an annihilating differential operator. The next two results give sufficient conditions for this kind of PDE barriers to identifiability to exist, the first for exponential families and the second for a broader class in which the score function is a polynomial in a finite collection of statistics.

\begin{definition}\label{def:exp_family}
The kernel $f$ belongs to an \emph{exponential family} with sufficient statistic of dimension $q\in\mathbb N$ if
\begin{equation}\label{eq:exp_family}
    f(x\mid\theta)=h(x)\exp\left\{\sum_{i=1}^q\eta_i(\theta)T_i(x)-A(\theta)\right\},\qquad x\in\mathcal X,\ \theta\in\Theta,
\end{equation}
where $h:\mathcal X\to(0,\infty)$ is measurable, $T=(T_1,\ldots,T_q):\mathcal X\to\mathbb R^q$ is a vector of sufficient statistics, $\eta=(\eta_1,\ldots,\eta_q):\Theta\to\mathbb R^q$ is the natural parameter map, and $A:\Theta\to\mathbb R$ is the log-partition function.
\end{definition}

\begin{proposition}\label{prop:exp_family_pde}
Let $f$ be an exponential family kernel as in \eqref{eq:exp_family}, with $\eta$ and $A$ infinitely differentiable on $\Theta^\circ$, and suppose that $q<d_2$. Then there exist a non-empty open set $U\subseteq\Theta^\circ$ on which the kernel family satisfies a non-trivial parameter PDE.
\end{proposition}

The proof of Proposition~\ref{prop:exp_family_pde} hinges upon the rank deficiency of the natural parameter map. Since $\eta$ takes values in $\mathbb R^{q}$ with $q<d_2$, its Jacobian $J_\eta(\theta)$, of dimension $q\times d_2$, has a non-trivial kernel at every $\theta$.  This makes it possible to construct a parameter differential operator that annihilates $f$ on an open subset of $\Theta$. 
%

Outside the class of exponential families, the score need not be linear in a finite set of statistics. As the next result shows, a polynomial dependence suffices to deduce the presence of a parameter PDE from a similar dimension mismatch.

\begin{definition}\label{def:polynomial_score}
The kernel $f$ has a \emph{polynomial score} of degree $p\ge1$ relative to $S=(S_1,\ldots,S_q):\mathcal X\to\mathbb R^q$ if $\theta\mapsto f(x\mid\theta)$ is infinitely differentiable on $\Theta^\circ$ for every $x\in\mathcal X$ and, for each $j\in\{1,\ldots,d_2\}$,
\begin{equation*}
    \frac{\partial_{\theta_j}f(x\mid\theta)}{f(x\mid\theta)}=P_j(S(x);\theta),\qquad x\in\mathcal X,\ \theta\in\Theta^\circ,
\end{equation*}
where $P_j(\cdot\,;\theta)$ is a polynomial in $q$ variables of total degree at most $p$ whose coefficients are infinitely differentiable in $\theta$ for almost every $x \in \mathcal{X}$. 
\end{definition}

\begin{proposition}\label{prop:general_pde}
Let $f$ have a polynomial score of degree $p$ relative to $S:\mathcal X\to\mathbb R^q$, and suppose that $q<d_2$. Then there exist a non-empty open set $U\subseteq\Theta^\circ$ on which the kernel family satisfies a non-trivial parameter PDE.
\end{proposition}

The proof of the previous result also relies on a comparison of dimensions, with the order of the annihilating operator governed by the mismatch between $d_2$ and $q$ and by the degree $p$. The previous result is similar in spirit to the exponential family case of Proposition~\ref{prop:exp_family_pde}, where $p=1$ and the dependence of $\log f$ on $S$ is linear.



\subsection{Non-identifiability via shift parameter PDEs}\label{sec:shift}

Several standard families are not subject to a parameter PDE in the sense of Definition~\ref{def:operator}, while they do satisfy a recurrence relation linking the kernel at $\theta$ with the kernel at translated parameter values, possibly combined with differentiation. While we postpone concrete examples, the following definition accommodates such relations.

\begin{definition}\label{def:shift_operator}
Let $V=\{v_1,\ldots,v_J\}\subset\mathbb R^{d_2}$ be a finite set of distinct shift vectors with $v_1=0$. Furthermore, let $r_1,\ldots,r_J\in\mathbb N_0$, and $U\subseteq\Theta^\circ$ be open and such that $U+v_j\subseteq\Theta^\circ$ for all $j$. A \emph{difference-differential parameter operator} $\mathcal L_\theta$ with shifts $V$ is defined to act on functions $ w\in C^{\max_{j}r_j}\big( \bigcup_j(U+v_j)\big)$ by
\begin{equation*}
    \mathcal L_\theta w=\sum_{j=1}^J\sum_{|\alpha|\le r_j}c_{j,\alpha}(\theta)\,(\partial^\alpha_\theta w)(\theta+v_j), \qquad \theta \in U,
\end{equation*}
with coefficients $c_{j,\alpha}\in C^{|\alpha|}(U)$. The operator is \emph{non-trivial} on a set $A\subseteq U$ if $c_{j,\alpha}$ is not identically zero on $A$ for at least one pair $(j,\alpha)$.
\end{definition}

Accordingly, we say that $f$ satisfies a \emph{shift parameter PDE} on $U$ if  there exists a difference-differential parameter operator $\mathcal L_\theta$, defined on some open $U\subseteq \Theta^\circ$ in the sense of Definition~\ref{def:shift_operator}, such that $\mathcal L_\theta f(x\mid\theta)=0$ for $\lambda$-almost every $x\in\mathcal X$ and every $\theta\in U$.

\begin{theorem}\label{thm:shift}
Let $\mathcal L_\theta$ be a difference-differential parameter operator with shifts $V=\{v_1,\ldots,v_J\}$ on an open set $U\subseteq\Theta^\circ$, and let $f$ satisfy the associated shift parameter PDE on $U$. Let $B=B(\theta_0,\rho)\subseteq U$ be an open ball such that $\overline{B+v_j}\subset\Theta^\circ$ for all $j =1 ,\ldots, J$ and whose radius $\rho$ satisifes $\rho< \frac{1}{2} \min_{i\ne j}\|v_i-v_j\|$.
%
%
Suppose that $\mathcal L_\theta$ is non-trivial on $B$. If the true mixing measure $G_*\in\mathcal P(\Theta)$ admits a density $g_*$ with respect to the Lebesgue measure on $\Theta$ satisfying $g_*\ge m$ Lebesgue-almost everywhere on $\bigcup_{j=1}^J(B+v_j)$, for some $m>0$, then there exists a $G\in\mathcal P(\Theta)$ with $G\ne G_*$ such that $p_G=p_{G_*}$ $\lambda$-almost everywhere.
\end{theorem}

The construction behind Theorem~\ref{thm:shift} parallels that of Theorem~\ref{thm:homogeneous}, whereby an appropriate test function is used to build a perturbation of $g_*$ that, by exploiting the non-trivial null space of the difference-differential parameter operator, induces non-identifiability of the mixing measure. 

\subsection{Examples}\label{sec:pde_examples}

We now discuss concrete kernel families for which the hypotheses of Theorems~\ref{thm:homogeneous} or \ref{thm:shift} are satisfied, and which therefore lead to non-identifiable infinite mixture models. 
The variety and popularity of these families are a testament to the pervasiveness of the PDE barriers to identifiability uncovered by our analysis. 

\paragraph{Location-scale Gaussian kernel.} With $\mathcal X=\mathbb R$ and $f(x\mid\mu,\nu)=\mathcal N(x\mid\mu,\nu)$, and for $\theta=(\mu,\nu)$ in a compact subset of $\mathbb R\times(0,\infty)$ with non-empty interior, the heat equation \eqref{eq:heat} gives a parameter PDE of order two with operator $\mathcal L_\theta=\partial_\nu-\frac12\partial^2_\mu$, so Theorem~\ref{thm:homogeneous} applies. This case was already discussed in Section~\ref{sec:contributions}.

\paragraph{Beta and Dirichlet kernels.} Let $\mathcal X=(0,1)$ and $f(x\mid\alpha,\beta)=[\Gamma(\alpha+\beta)/\Gamma(\alpha)\Gamma(\beta)]x^{\alpha-1}(1-x)^{\beta-1}$, the beta kernel, with $(\alpha,\beta)$ ranging in a compact subset $\Theta$ of $(0,\infty)^2$. Since $\alpha(\alpha+\beta)^{-1}f(x\mid\alpha+1,\beta)=xf(x\mid\alpha,\beta)$ and $\beta(\alpha+\beta)^{-1}f(x\mid\alpha,\beta+1)=(1-x)f(x\mid\alpha,\beta)$ for all $x\in\mathcal X$ and $(\alpha,\beta)\in\Theta$, we have
\begin{equation*}
    \frac{\alpha}{\alpha+\beta}f(x\mid\alpha+1,\beta)+\frac{\beta}{\alpha+\beta}f(x\mid\alpha,\beta+1)-f(x\mid\alpha,\beta)=0,
\end{equation*}
a shift parameter PDE of order zero with shifts $(0,0)$, $(1,0)$ and $(0,1)$. The same argument applies to the more general Dirichlet family on the open simplex $\Delta_{d-1}=\{x\in\mathbb R^{d}:x_i>0,\ \sum_{i=1}^dx_i=1\}$ with density $f(x\mid\alpha)=\Gamma(\sum_{i=1}^d\alpha_i)\prod_{i=1}^d(x_i^{\alpha_i-1}/\Gamma(\alpha_i))$, for which
\begin{equation*}
    \sum_{j=1}^d\frac{\alpha_j}{\sum_{m=1}^d\alpha_m}\,f(x\mid\alpha+e_j)-f(x\mid\alpha)=0,
\end{equation*}
with shifts $0,e_1,\ldots,e_d$.

\paragraph{Location-scale Student-$t$ kernel with fixed degrees of freedom.} Let $\mathcal X=\mathbb R$, let $\nu>0$ be fixed, and let
\begin{equation*}
    f(x\mid\mu,\sigma)=\frac{\Gamma((\nu+1)/2)}{\Gamma(\nu/2)(\nu\pi)^{1/2}\sigma}\left[1+\frac1\nu\left(\frac{x-\mu}{\sigma}\right)^2\right]^{-(\nu+1)/2},
\end{equation*}
with $\theta=(\mu,\sigma)$ in a compact subset of $\mathbb R\times(0,\infty)$. Direct computation gives the parameter PDE of order two
\begin{equation*}
    \big[\sigma\partial^2_\sigma+\nu\sigma\partial^2_\mu+(1-\nu)\partial_\sigma\big]f(x\mid\mu,\sigma)=0,
\end{equation*}
so that Theorem~\ref{thm:homogeneous} applies.

\paragraph{Gamma kernel.} Let $\mathcal X=(0,\infty)$ and $f(x\mid\gamma,\beta)=\beta^\gamma x^{\gamma-1}e^{-\beta x}/\Gamma(\gamma)$, with $\theta =(\gamma,\beta)$ in a compact subset of $(0,\infty)^2$. Using the identities $xf(x\mid\gamma,\beta)=\gamma\beta^{-1}f(x\mid\gamma+1,\beta)$ and $\partial_\beta f=\gamma\beta^{-1}f-xf$ gives
\begin{equation*}
    \frac{\beta}{\gamma}\,\partial_\beta f(x\mid\gamma,\beta)-f(x\mid\gamma,\beta)+f(x\mid\gamma+1,\beta)=0,
\end{equation*}
a shift parameter PDE with shifts $v_1=(0,0)$ and $v_2=(1,0)$ and orders $r_1=1$, $r_2=0$; Theorem~\ref{thm:shift} applies whenever the parameter space accommodates a ball together with its translate by $v_2$.

\paragraph{Negative binomial kernel with varying number of successes.} Let $\mathcal X=\mathbb N_0$, let $\lambda$ be the counting measure, and let $f(x\mid r,p)=\frac{\Gamma(x+r)}{x!\Gamma(r)}p^r(1-p)^x$, with $(r,p)$ in a compact subset of $(0,\infty)\times(0,1)$. Then
\begin{equation*}
    f(x\mid r+1,p)+\frac{p(1-p)}{r}\,\partial_pf(x\mid r,p)-f(x\mid r,p)=0,
\end{equation*}
a shift parameter PDE with shifts $(0,0)$ and $(1,0)$.

\paragraph{Non-central chi-squared kernel.} Let $\mathcal X=(0,\infty)$ and
\begin{equation*}
    f(x\mid k,\lambda_{\mathrm{nc}})=\tfrac12e^{-(x+\lambda_{\mathrm{nc}})/2}\left(\frac{x}{\lambda_{\mathrm{nc}}}\right)^{k/4-1/2}I_{k/2-1}\big((\lambda_{\mathrm{nc}}x)^{1/2}\big),
\end{equation*}
where $I_\kappa(z)=\sum_{j\ge0}(z/2)^{2j+\kappa}/[j!\,\Gamma(j+\kappa+1)]$ is the modified Bessel function of the first kind, the degrees of freedom $k$ and the non-centrality parameter $\lambda_{\mathrm{nc}}$ ranging over a compact subset of $(0,\infty)^2$. The classical recurrence
\begin{equation*}
    \partial_{\lambda_{\mathrm{nc}}}f(x\mid k,\lambda_{\mathrm{nc}})-\tfrac12f(x\mid k+2,\lambda_{\mathrm{nc}})+\tfrac12f(x\mid k,\lambda_{\mathrm{nc}})=0
\end{equation*}
is a shift parameter PDE with shifts $(0,0)$ and $(2,0)$.
\subsection{Consequences for statistical estimation}\label{sec:statistical_estimation}

Having established that a wide range of kernel families induce
non-identifiable infinite mixture models, we now briefly discuss the consequences of
this phenomenon for statistical estimation. Since the large majority of
mixtures used in practice model the mixing measure as discrete, the following
result focuses on that case, although considering general members of
$\mathcal P(\Theta)$ would not alter the conclusions.
\begin{proposition}
\label{pro:statistical_estimation}
Let $\{f(\cdot\mid \theta) : \theta\in\Theta\}$ be a kernel family such that the mapping $\theta\mapsto f(x\mid\theta)$ is continuous on $\Theta$ for $\lambda$-almost every $x$. If there exist two mixing measures $G_1,G_2\in\mathcal P(\Theta)$ with $G_1\neq G_2$ and
$p_{G_1}=p_{G_2}$ $\lambda$-almost everywhere, then
$$ \inf_{n\in\mathbb N} \,\inf_{\hat{G}_n}\, \sup_{G_* \in \mathcal{D}(\Theta)} \mathbb{E}_{X_{1:n}\sim P_{G_{*}}^n} \big[ W_1(\widehat{G}_n(X_{1:n}), G_{*}) \big] \geq \frac{3c}{32},$$
where $c = W_{1}(G_{1}, G_{2})$ and the inner infimum is taken over all estimators $\widehat G_n:\mathcal X^n\to\mathcal D(\Theta)$.
\end{proposition}


The proof of Proposition~\ref{pro:statistical_estimation} relies on standard
tools from minimax theory together with the $W_1$-density of
$\mathcal D(\Theta)$ in $\mathcal P(\Theta)$. The above results shows that fully unconstrained
estimation of the mixing measure is inconsistent, in the worst-case sense,
whenever identifiability fails. In particular, it shows that our non-identifiability results,
though stated for mixing measures with a locally positive density with respect
to the Lebesgue measure on $\Theta$, carry strong implications for estimation
even in the seemingly more structured class of discrete
mixtures. Hence, unless further knowledge of the structure of $G_*$ is
available and used to constrain estimation---how to do so being an open
problem---recovery of $G_*$ from a growing sample is made impossible, in
the worst case, by the non-identifiability phenomena uncovered by our theory.

\section{Identifiability-preserving kernel families}\label{sec:identifiable}

The results of Section~\ref{sec:pde_barriers} identify specific differential and difference-differential structures in the kernel parameterisation as mechanisms impeding the recovery of the underlying mixing measure from a given mixture density. As a complement to these negative results, we now describe three classes of kernels for which the mixture operator is injective, which ensures identifiability and restores the hope for consistent unconstrained estimation of $G_*$. 

\subsection{Generalized translation families}\label{sec:translation}

The first such class consists of kernels whose dependence on the parameter is, after a change of variables, that of a location family.

\begin{definition}\label{def:translation}
The kernel $f$ belongs to a \emph{generalized translation family} if
\begin{equation}\label{eq:translation}
    f(x\mid\theta)=h(x)\,\gamma(\theta)\,K(\tau(x)-\psi(\theta)),\qquad x\in\mathcal X,\ \theta\in\Theta,
\end{equation}
where $h:\mathcal X\to(0,\infty)$ and $\gamma:\Theta\to(0,\infty)$ are continuous, $\tau:\mathcal X\to\mathbb R^{d_1}$ is a homeomorphism onto $\mathbb R^{d_1}$, $\psi:\Theta\to\mathbb R^{d_1}$ is a homeomorphism onto its image, and $K\in L^1(\mathbb R^{d_1})$.
\end{definition}

\begin{proposition}\label{prop:translation}
Let $f$ belong to a generalized translation family and suppose that $\mathcal F[K](\xi)\ne0$ for Lebesgue-almost every $\xi\in\mathbb R^{d_1}$. Then $\mathcal A$ is injective.
\end{proposition}

The structure of generalized translation classes reduces the proof of
Proposition~\ref{prop:translation} to a classical Fourier-analytic
deconvolution argument. Moreover, the non-vanishing requirement on $\mathcal F[K]$
is mild. Indeed, suppose $d_1=1$ and $\int_{\mathbb R}|K(x)|e^{c|x|}\,dx<\infty$ for
some $c>0$. Then $\mathcal F[K]$ extends to a holomorphic function on the
strip $\{z\in\mathbb C:|\mathrm{Im}(z)|<c\}$. A set of positive Lebesgue
measure on $\mathbb R$ has an accumulation point in $\mathbb R$, which lies
inside the strip; were $\mathcal F[K]$ to vanish on such a set, the identity
theorem would force $\mathcal F[K]=0$ on the strip and hence $K=0$ almost
everywhere. Since by construction $K\ne0$ on a set of positive Lebesgue measure, its transform $\mathcal F[K]$ vanishes at most on a
Lebesgue-null set, so any integrable kernel with exponentially decaying tails
satisfies the hypothesis of Proposition~\ref{prop:translation}. Analogous
arguments apply in higher dimension.

The following standard kernels belong to generalized translation families, and the resulting infinite mixture models are therefore identifiable. 

\begin{enumerate}
    \item[(i)] \emph{Gaussian kernel with fixed variance.} With $\mathcal X=\mathbb R$ and $f(x\mid\theta)=\mathcal N(x\mid\theta,\sigma^2)$ for fixed $\sigma^2>0$, representation \eqref{eq:translation} holds with $h(x)=(2\pi\sigma^2)^{-1/2}$, $\gamma=1$, $\tau(x)=x$, $\psi(\theta)=\theta$ and $K(u)=\exp\{-u^2/(2\sigma^2)\}$, whose transform $\mathcal F[K](\xi)=(2\pi)^{1/2}\sigma\exp(-\sigma^2\xi^2/2)$ is strictly positive.

    \item[(ii)] \emph{Gamma kernel with fixed shape.} With $\mathcal X=(0,\infty)$
and $f(x\mid\theta)=[\Gamma(\alpha)\theta^{\alpha}]^{-1}x^{\alpha-1}e^{-x/\theta}$
for fixed $\alpha>0$, take $h(x)=[\Gamma(\alpha)x]^{-1}$, $\gamma(\theta)=1$,
$\tau(x)=\log x$, $\psi(\theta)=\log\theta$ and $K(u)=e^{\alpha u}\exp(-e^{u})$.
Here $\tau$ maps $\mathcal X$ onto $\mathbb R$, and $K$ decays exponentially at
$-\infty$ and doubly exponentially at
$+\infty$, so $K\in L^1(\mathbb R)$ has exponentially decaying tails and the
remarks after Proposition~\ref{prop:translation} apply.

    \item[(iii)] \emph{Laplace kernel with fixed scale.} With $\mathcal X=\mathbb R$ and $f(x\mid\theta)=(2b)^{-1}\exp(-|x-\theta|/b)$ for fixed $b>0$, take $h=1$, $\gamma=1$, $\tau(x)=x$, $\psi(\theta)=\theta$ and $K(u)=(2b)^{-1}e^{-|u|/b}$, whose transform $\mathcal F[K](\xi)=(1+b^2\xi^2)^{-1}$ is everywhere positive.
\end{enumerate}

The above examples show that fixing an appropriate subset of parameters, such as
the scale or the shape, removes the PDE barriers documented in
Section~\ref{sec:pde_barriers} for the analogous kernels with more free
parameters. Moreover, as already noted, location families such as (i) and (iii)
are precisely those on which mixing-measure estimation has largely concentrated
so far \citep[among others]{Nguyen-13,Gao2016Posterior}. In hindsight, our results suggest that
it is exactly their identifiability, absent in the more flexible (e.g., location-scale)
alternatives, that has made the recent technical progress on estimation rates possible in the
corresponding infinite mixture models.


\subsection{Non-degenerate exponential family}\label{sec:nondegenerate}

Proposition~\ref{prop:exp_family_pde} shows that an exponential family whose parameter dimension exceeds the dimension $q$ of its sufficient statistic carries a parameter PDE and therefore yields non-identifiable infinite mixture models. We now consider the complementary regime, in which the sufficient statistic has dimension at least as high as the parameter vector.

\begin{definition}\label{def:determining}
Let $\Omega\subset(0,\infty)^q$ be compact. A set $S\subseteq\mathbb R^q$ is \emph{determining for $\Omega$} if the linear span of $\{y\mapsto y^{s}:s\in S\}$, where $y^{s}:=y_1^{s_1}\cdots y_q^{s_q}$ for $y\in(0,\infty)^q$, is dense in the space $C(\Omega)$ of continuous real functions on $\Omega$ equipped with the supremum norm.
\end{definition}

\begin{proposition}\label{prop:exp_family_identif}
Let $f$ be an exponential family kernel as in \eqref{eq:exp_family} such that $h>0$ $\lambda$-almost everywhere, $A$ is continuous on the compact parameter space $\Theta$, $T$ is continuous on $\mathcal X$, and  $\eta$ is injective and continuous (hence a homeomorphism onto its image $\eta(\Theta)$. If $T(\mathcal X)$ is determining for $\exp\{\eta(\Theta)\}$, where the exponential is applied coordinate-wise, then $\mathcal A$ is injective.
\end{proposition}

Definition~\ref{def:determining} is a Müntz--Szász-type condition. For $q=1$ and $S=\{s_1,s_2,\ldots\}$, the classical theorem of \cite{Muntz1914} characterizes density in terms of the divergence of $\sum_{n\in\mathbb N}s_n^{-1}$. Multivariate versions of this statement are available, though the picture is less complete \citep{Kroo1994GeometricMuntz,yang2014multivariate}. The condition is easy to verify in many cases of interest. For instance, if $S$ contains $\mathbb N_0^q$, as when $\mathcal X=\mathbb N_0^{q}$ and $T$ is the identity, then the span is the algebra of polynomials, which separates points in $\Omega$ and contains the constants, so the density property follows from the Stone--Weierstrass theorem. The same holds when $T(\mathcal X)=\mathbb R^{q}$. Moreover, notice that injectivity of $\eta$ on a compact set with non-empty interior in $\mathbb R^{d_2}$ requires $q\ge d_2$, so the proposition deals with the regime complementary to that of Proposition~\ref{prop:exp_family_pde}. 

The following families satisfy the hypotheses of Proposition~\ref{prop:exp_family_identif}:

\begin{enumerate}
    \item[(i)] \emph{Poisson kernel.} With $\mathcal X=\mathbb N_0$, $\lambda$ the counting measure, $f(x\mid\theta)=\theta^xe^{-\theta}/x!$ and $\Theta$ a compact subset of $(0,\infty)$, take $h(x)=1/x!$, $T(x)=x$, $\eta(\theta)=\log\theta$ and $A(\theta)=\theta$. Here $q=d_2=1$, $\eta$ is a homeomorphism onto its image, and $T(\mathcal X)=\mathbb N_0$ is determining for $\exp(\eta(\Theta))=\Theta$ by the previous remarks;

    \item[(ii)] \emph{Negative binomial kernel with fixed number of successes.} With $\mathcal X=\mathbb N_0$, $f(x\mid\theta)=\binom{x+r-1}{x}(1-\theta)^x\theta^r$ for fixed $r>0$ and $\Theta$ a compact subset of $(0,1)$, take $h(x)=\binom{x+r-1}{x}$, $T(x)=x$, $\eta(\theta)=\log(1-\theta)$ and $A(\theta)=-r\log\theta$, so that $T(\mathcal X)=\mathbb N$ is determining for $\exp(\eta(\Theta))=1-\Theta$ by the previous remarks;

    \item[(iii)] \emph{Multivariate Gaussian kernel with fixed covariance.} With $\mathcal X=\mathbb R^d$, $$f(x\mid\theta)=(2\pi)^{-d/2}|\Sigma|^{-1/2}\exp\left\{-\frac12(x-\theta)^\top\Sigma^{-1}(x-\theta)\right\}$$ for a fixed positive definite $\Sigma$ and $\Theta$ a compact subset of $\mathbb R^d$, take $\eta(\theta)=\theta$, $h(x)=\exp(-\frac12x^\top\Sigma^{-1}x) (2\pi)^{-d/2} |\Sigma|^{-1/2}$, $T(x)=\Sigma^{-1}x$, and $A(\theta)=\frac12\theta^\top\Sigma^{-1}\theta$. Here $q=d_2=d$ and $T(\mathcal X)=\mathbb R^d$ is obviously determining for $\exp(\eta(\Theta))=\exp(\Theta)$.
\end{enumerate}

\subsection{Beyond the determining set assumption}\label{sec:beyond}

The determining set condition of Proposition~\ref{prop:exp_family_identif} constrains the range of the sufficient statistic. We conclude by showing, in an example, that identifiability can hold without it, provided the sufficient statistics satisfy appropriate conditions. 
Specifically, let $\mathcal X=\mathbb R$, so $d_1=1$, and let $\Theta=[-a,a]\times[-b,b]\subset\mathbb R^2$ with $a,b>0$, so $d_2=2$. Consider the kernel
\begin{equation}\label{eq:super_exponential}
    f(x\mid\theta_1,\theta_2)=h(x)\exp\big\{\theta_1x+\theta_2e^{x^2}-A(\theta)\big\},\qquad h(x)=\exp\big(-2be^{x^2}\big),
\end{equation}
which we call a \emph{super-exponential kernel}. Its sufficient statistics are $T_1(x)=x$ and $T_2(x)=e^{x^2}$, and the choice of base measure guarantees integrability: since $\theta_2-2b\le-b<0$ on $\Theta$, the exponent $\theta_1x+(\theta_2-2b)e^{x^2}$ is bounded above by $\theta_1x-be^{x^2}$, so $A$ is finite, continuous and bounded on the compact $\Theta$.

\begin{proposition}\label{prop:super_exponential}
Let $f$ be the kernel \eqref{eq:super_exponential}. Then,
\begin{enumerate}
    \item[(a)] the only parameter differential operator of finite order, with continuous coefficients on $\Theta^\circ$, that annihilates $f$ is the trivial one, so that $f$ satisfies no non-trivial parameter PDE;
    \item[(b)] $\mathcal A$ is injective.
\end{enumerate}
\end{proposition}

Both statements follow from a mismatch in growth between the two sufficient statistics. 
Moreover, we emphasize that the construction is not specific to the statistic $e^{x^2}$: the same argument applies to kernels of the form $f(x\mid\theta_1,\theta_2)=h(x)\exp\{\theta_1x+\theta_2e^{P(x)}-A(\theta)\}$ for any polynomial $P$ of even degree with positive leading coefficient (with the base measure adjusted accordingly). 

\begin{remark}
    Proposition~\ref{prop:super_exponential} exhibits a kernel family for which (a) the absence of a parameter PDE and (b) identifiability of the mixing measure coexist. Whether the absence of these or similar differential structures is sufficient for identifiability in more general scenarios, and not just  necessary for it as discussed in Section~\ref{sec:pde_barriers}, remains an open question. 
\end{remark}

\section{Discussion}\label{sec:discussion}

We have identified differential structure in the parameterisation of the kernel
as a fundamental barrier to the identifiability of the mixing measure in
infinite mixture models. When the kernel is annihilated by a non-trivial
parameter differential operator, or by a difference-differential operator
involving shifted parameter values, distinct mixing measures induce the same
mixture density, and the mixing measure cannot be recovered in general. Such operators have been shown to exist for a wide range of standard
families, where over-parameterisation emerges as a
systematic route to non-identifiability. Complementing these barriers, we have
exhibited three classes of kernels for which the mixture operator is injective,
showing that the phenomenon is not universal and that identifiability can be
recovered by suitably restricting the parameterisation.

From our analysis, two main questions emerge as worthy of future
investigation. The first concerns the exact scope of the differential barrier
mechanism. Our results establish that the presence of a parameter PDE is
sufficient for non-identifiability, but not that its absence is sufficient for
identifiability. The super-exponential family of Section~\ref{sec:beyond} shows
that the two properties can indeed coexist: there, no non-trivial parameter PDE
exists and, at the same time, the mixture operator is injective. Whether the absence of
differential structure is sufficient for identifiability in general, or whether
further obstructions may arise, remains to be understood.

The second question concerns the implications of our results for statistical
estimation of the mixing measure. Proposition~\ref{pro:statistical_estimation}
shows that, when identifiability fails, no estimator can recover the mixing
measure in the worst case, so that our findings make the study of consistent and fully unconstrained nonparametric estimation hopeless. This worst-case analysis,
however, leaves open the question of which constraints on the model class or on
the true mixing measure may rule out the perturbations underlying our
constructions and restore consistent estimation. For instance, the recent
results of \cite{bariletto2026convergence}, which establish convergence rates
for the mixing measure in location-scale mixtures with a scale parameter shared
across components and learned from the data, may be read precisely as imposing one such constraint on the
model class and the ground truth, i.e., that the scale coordinate be held common. A more systematic understanding of which constraints, of this or other
kinds, suffice to restore identifiability would give a more nuanced picture of statistical estimation in infinite mixture models.

\clearpage
\begin{center}

{\bf{\LARGE{Supplementary Material}}}
\end{center}

\appendix

In this Supplementary Material, we collect the proofs of all the results presented in the main text of the article (Appendix~\ref{app:proofs_main}), and the statements and proofs of auxiliary results (Appendix~\ref{app:auxiliary}).

\section{Proofs of the results in the main text}
\label{app:proofs_main}

Throughout this appendix we use the notation of Sections~\ref{sec:notation}--\ref{sec:identifiable}.
For a parameter differential operator $\mathcal L_\theta=\sum_{|\alpha|\le r}c_\alpha(\theta)\partial^\alpha_\theta$
on an open set $U\subseteq\Theta^\circ$, recall its formal adjoint
$\mathcal L^*_\theta w=\sum_{|\alpha|\le r}(-1)^{|\alpha|}\partial^\alpha_\theta(c_\alpha w)$,
and the pairing $\langle w_1,w_2\rangle=\int_\Theta w_1(\theta)w_2(\theta)\,d\theta$.

\subsection{Proof of Theorem~\ref{thm:homogeneous}}
\label{app:proof:thm_homogeneous}


\emph{Step 1.}
For $u\in C_c^\infty(B)$ and $v\in C^r(U)$, integration by parts on $B$, with no
boundary contributions since $u$ and its derivatives vanish near $\partial B$, gives
\begin{equation*}
    \langle \mathcal L_\theta v,\,u\rangle
    =\sum_{|\alpha|\le r}\int_B c_\alpha(\theta)\big(\partial^\alpha_\theta v(\theta)\big)u(\theta)\,d\theta
    =\sum_{|\alpha|\le r}(-1)^{|\alpha|}\int_B v(\theta)\,\partial^\alpha_\theta\!\big(c_\alpha(\theta)u(\theta)\big)\,d\theta
    =\langle v,\,\mathcal L^*_\theta u\rangle,
\end{equation*}
where $\bar h:=\mathcal L^*_\theta u$ is continuous and compactly supported in $B$.

\vspace{0.5 em}
\noindent
\emph{Step 2.}
We first show that $\mathcal L_\theta u=0$ for all $u\in C_c^\infty(B)$ is impossible.
Since $\mathcal L_\theta$ is non-trivial on $B$, there is $\theta_*\in B$ and a multi-index
$\beta$ with $|\beta|\ge1$ and $c_\beta(\theta_*)\ne0$. Choose $\varphi\in C_c^\infty(B)$
equal to $1$ in a neighborhood of $\theta_*$ and set $u_\beta(\theta)=\frac{1}{\beta!}
\varphi(\theta)(\theta-\theta_*)^\beta$,
so that $\partial^\alpha_\theta u_\beta(\theta_*)=\mathbf 1\{\alpha=\beta\}$. Then,
$\mathcal L_\theta u_\beta(\theta_*)=\sum_{|\alpha|\le r}c_\alpha(\theta_*)\partial^\alpha_\theta u_\beta(\theta_*)=c_\beta(\theta_*)\ne0$. Consequently, if $\mathcal L^*_\theta u=0$ for every $u\in C_c^\infty(B)$, then
$\langle \mathcal L_\theta v,u\rangle=\langle v,\mathcal L^*_\theta u\rangle=0$ for all $u,v\in C_c^\infty(B)$,
forcing $\mathcal L_\theta v\equiv0$ on $B$ for all $v$, which was shown to be impossible. Hence there exists
$u\in C_c^\infty(B)$ with $\bar h:=\mathcal L^*_\theta u\not\equiv0$.

\vspace{0.5 em}
\noindent
\emph{Step 3.}
By Fubini's theorem and $\int_{\mathcal X}f(x\mid\theta)\lambda(dx)=1$,
\begin{equation*}
    \int_\Theta \bar h(\theta)\,d\theta
    =\int_{\mathcal X}\left(\int_\Theta \bar h(\theta)f(x\mid\theta)\,d\theta\right)\lambda(dx),
\end{equation*}
while the parameter PDE and the adjoint identity give, for $\lambda$-almost every $x$,
\begin{equation*}
    \int_\Theta f(x\mid\theta)\bar h(\theta)\,d\theta
    =\langle f(x\mid\cdot),\mathcal L^*_\theta u\rangle
    =\langle \mathcal L_\theta f(x\mid\cdot),u\rangle=0 .
\end{equation*}
Hence $\int_\Theta\bar h\,d\theta=0$ and $\int_\Theta f(x\mid\theta)\bar h(\theta)\,d\theta=0$ for $\lambda$-almost every $x$.

\vspace{0.5 em}
\noindent
\emph{Step 4.}
Let $M:=\sup_{\theta\in\Theta}|\bar h(\theta)|\in(0,\infty)$ and set $\varepsilon:=m/(2M)>0$.
Define $g:=g_*+\varepsilon\bar h$. On $B$ we have $g\ge m-\varepsilon M=m/2>0$, and on $\Theta\setminus B$
we have $g=g_*\ge0$, since $\bar h$ is supported in $B$; thus $g\ge0$ almost everywhere. Moreover
$\int_\Theta g\,d\theta=\int_\Theta g_*\,d\theta+\varepsilon\int_\Theta\bar h\,d\theta=1$, so $g$ is the
density of a probability measure $G\in\mathcal P(\Theta)$. Finally, for $\lambda$-almost every $x$,
\begin{equation*}
    p_G(x)-p_{G_*}(x)=\varepsilon\int_\Theta f(x\mid\theta)\bar h(\theta)\,d\theta=0.
\end{equation*}
At the same time, because $\bar h$ is non-zero at $\theta_*$ and continuous on $B$, it is non-zero on a set of positive Lebesgue measure. Therefore $G\ne G_*$, proving the claim.

\subsection{Proof of Proposition~\ref{prop:exp_family_pde}}
\label{app:proof:prop_exp_family_pde}

Write $f(x\mid\theta)=h(x)\exp\{\sum_{i=1}^q\eta_i(\theta)T_i(x)-A(\theta)\}$.
Differentiating in $\theta_j$, we obtain
\begin{equation*}
    \partial_{\theta_j}f(x\mid\theta)=f(x\mid\theta)\left(\sum_{i=1}^q\partial_{\theta_j}\eta_i(\theta)\,T_i(x)-\partial_{\theta_j}A(\theta)\right).
\end{equation*}
For a first-order operator $\mathcal L_\theta=\sum_{j=1}^{d_2}c_j(\theta)\partial_{\theta_j}+c_0(\theta)$, we get
\begin{equation*}
    \mathcal L_\theta f(x\mid\theta)
    =f(x\mid\theta)\left[\sum_{i=1}^q T_i(x)\sum_{j=1}^{d_2}c_j(\theta)\partial_{\theta_j}\eta_i(\theta)
    -\sum_{j=1}^{d_2}c_j(\theta)\partial_{\theta_j}A(\theta)+c_0(\theta)\right].
\end{equation*}
Since $f>0$, the operator annihilates $f$ on an open set $U\subseteq \Theta^\circ$ so long as
\begin{equation}
    \sum_{j=1}^{d_2}c_j(\theta)\partial_{\theta_j}\eta_i(\theta)=0\ \ (1\le i\le q)
    \qquad \text{and} \quad 
    c_0(\theta)=\sum_{j=1}^{d_2}c_j(\theta)\partial_{\theta_j}A(\theta)
    \label{eq:app_exp_1},
\end{equation}
for all $\theta\in U$. Below, we will demonstrate a choice of $U$ and of the functions $(c_0(\theta), c_1(\theta), \ldots, c_{d_2}(\theta))$ for each $\theta \in U$ satisfying the above relationship.

First, notice that by writing $J_\eta(\theta)$ for the $q\times d_2$ Jacobian with $(i,j)$-th entry $[J_\eta(\theta)]_{ij}=\partial_{\theta_j}\eta_i(\theta)$
and letting $\mathbf c(\theta)=(c_1(\theta),\dots,c_{d_2}(\theta))^\top$, the first set of equations is equivalent to
\begin{equation}\label{eq:JC}
J_\eta(\theta)\mathbf c(\theta)=0.    
\end{equation}

Let $r := \max_{\theta \in \Theta^\circ} \operatorname{rank}(J_\eta(\theta))$ 
and 
$R := \{\theta \in \Theta^\circ : \operatorname{rank}(J_\eta(\theta)) = r\}$.  By construction, for any $\theta' \in R$, there exists  some $r \times r$ submatrix $M(\theta')$ of $J_\eta(\theta')$
such that $\det M(\theta') \ne 0$. Without loss of generality, we can assume that $M(\theta')$ is the principal submatrix (i.e. it occupies the first $r$ rows and the first $r$ columns) of the Jacobian $J_{\eta}(\theta')$. Since the entries $\partial_{\theta_j}\eta_i$ are continuous and the
determinant is a multilinear polynomial in those entries (and therefore continuous), there exists a $\delta > 0$ such that  $U := B(\theta',\delta) \subset R$ and, for each $\theta$ in $U$,  $ \operatorname{rank}(J_\eta(\theta)) = \operatorname{rank}(M(\theta)) = r$, where $M(\theta)$ is the leading principal submatrix of $J_\eta(\theta)$ of dimension $r \times r$.

Write $\mathbf c_A(\theta) = (c_1(\theta), \dots, c_r(\theta))^\top$ and
$\mathbf c_B(\theta) = (c_{r+1}(\theta), \dots, c_{d_2}(\theta))^\top$. Because
$d_2 - r \ge d_2 - q > 0$,  $\mathbf c_B(\theta)$ contains at least one coordinate. Let $B(\theta)$
denote the $r \times (d_2 - r)$ submatrix of $J_\eta(\theta)$ adjacent to $M(\theta)$, and let $\mathbf b_1(\theta)$
be its first column.
We now set $\mathbf c_B(\theta)$ as the one-shot vector $(1, 0, \dots, 0)^\top$ for every $\theta \in U$ and show that a solution to \eqref{eq:JC} exists for every $\theta \in U$. Indeed, the identity $J_\eta(\theta)\mathbf c(\theta) = 0$ becomes
$$ M(\theta)\mathbf c_A(\theta) + B(\theta)\mathbf c_B(\theta) = 0, $$
which in turn is equivalent to
$$ M(\theta)\mathbf c_A(\theta) = -B(\theta)\mathbf c_B(\theta) = -\mathbf b_1(\theta). $$
Since $\det(M(\theta)) \ne 0$ on $U$, the above system has a unique set of solutions  $(c_j(\theta), 1 \le j \le r)$ for each $\theta \in U$, given by Cramer's rule as
$$ c_j(\theta) = \frac{\det(M_j(\theta))}{\det(M(\theta))}, $$
where $M_j(\theta)$ is formed by replacing the $j$-th column of $M(\theta)$ with the vector
$-\mathbf b_1(\theta)$. From this choice of $\mathbf c(\theta)$ it follows that
\begin{equation}
\label{eqn:prop_2_dependent_when_i_less_than_r}
    \sum_{j=1}^{d_2} c_j(\theta)\frac{\partial \eta_i(\theta)}{\partial \theta_j} = 0
    \qquad \text{for all } 1 \le i \le r.
\end{equation}
We now verify that equation \eqref{eqn:prop_2_dependent_when_i_less_than_r} holds also for $i\geq r$. Let $\mathbf r_1(\theta), \dots, \mathbf r_q(\theta)$ be the row vectors of $J_\eta(\theta)$. For $i > r$, since $\operatorname{rank}(J_\eta(\theta)) = r$ and $M(\theta)$ is invertible on $U$, for each $\theta \in U$
there exist real numbers $\lambda^i_1(\theta), \dots, \lambda^i_r(\theta)$ such that
\begin{equation*}
    \mathbf r_i(\theta) = \sum_{m=1}^{r} \lambda^i_m(\theta)\,\mathbf r_m(\theta).
\end{equation*}
As a result, for every $i > r$,
\begin{align}
\label{eqn:prop_2_dependent_when_i_larger_than_r}
    \sum_{j=1}^{d_2} c_j(\theta)\frac{\partial \eta_i(\theta)}{\partial \theta_j}
    = \langle \mathbf c(\theta), \mathbf r_i(\theta)\rangle
    = \sum_{m=1}^{r} \lambda^i_m(\theta)\,\langle \mathbf c(\theta), \mathbf r_m(\theta)\rangle = 0,
\end{align}
the final equality following from \eqref{eqn:prop_2_dependent_when_i_less_than_r}. Combining
\eqref{eqn:prop_2_dependent_when_i_less_than_r} and \eqref{eqn:prop_2_dependent_when_i_larger_than_r}, we
verify that \eqref{eq:JC} holds on $U$.

To complete the proof, it remains to show the second identity in \eqref{eq:app_exp_1}.
Towards that goal, we set, for each $\theta \in U$,
$$ c_0(\theta) := \sum_{j=1}^{d_2} c_j(\theta)\frac{\partial A(\theta)}{\partial \theta_j}, $$
where $(c_j(\theta), 1 \le j \le d_2)$ are the coefficient functions derived above. 
Since the entries of the Jacobian are infinitely continuously differentiable, both $\theta \in U \mapsto \det(M(\theta))$ and $\theta \in U \mapsto \det(M_j(\theta))$ are as well. As $\det(M(\theta))$ is bounded away from zero on $U$, we conclude  that 
$c_j \in C^\infty(U)$ for all $j \geq 1$. This, along with the fact that $A \in C^\infty(\Theta^\circ)$, yields that $c_0 \in C^\infty(U)$, too. This completes the proof.


\subsection{Proof of Proposition~\ref{prop:general_pde}}
\label{app:proof:prop_general_pde}

\emph{Step 1.}
We prove by induction on $m=|\alpha|$ that, for every multi-index $\alpha\in\mathbb N_0^{d_2}$,
\begin{equation}
    \partial^\alpha_\theta f(x\mid\theta)=Q_\alpha\big(S(x);\theta\big)\,f(x\mid\theta),
    \label{eq:app_gen_1}
\end{equation}
where $Q_\alpha(\,\cdot\,;\theta)$ is a polynomial in the $q$ variables $S_1(x),\dots,S_q(x)$ with
coefficients that
lie in $C^\infty(\Theta^\circ)$ and of total degree $\deg Q_\alpha\le p\,|\alpha|$.

For $|\alpha|=0$, \eqref{eq:app_gen_1} holds with $Q_0\equiv1$. For $|\alpha|=1$, say $\alpha=e_j$,
Definition~\ref{def:polynomial_score} gives $\partial_{\theta_j}f(x\mid\theta)=P_j\big(S(x);\theta\big)f(x\mid\theta)$
with $\deg P_j\le p$, so \eqref{eq:app_gen_1} holds with $Q_{e_j}=P_j$; moreover, by assumption the coefficients
of $P_j$ are in $C^\infty(\Theta^\circ)$.

Assume \eqref{eq:app_gen_1} holds for all multi-indices of order $m$, and take $\beta=\alpha+e_j$, for some $\alpha$ with $|\alpha| = m$ and $j$, so that 
$|\beta| = m+1$. Then
\begin{equation*}
    \partial^\beta_\theta f
    =\partial_{\theta_j}\big(\partial^\alpha_\theta f\big)
    =\partial_{\theta_j}\big[Q_\alpha\big(S(x);\theta\big)f(x\mid\theta)\big]
    =\Big(\partial_{\theta_j}Q_\alpha\Big)f+Q_\alpha\big(\partial_{\theta_j}f\big)
    =\big(\partial_{\theta_j}Q_\alpha+Q_\alpha P_j\big)f .
\end{equation*}
Define $Q_\beta:=\partial_{\theta_j}Q_\alpha+Q_\alpha P_j$. Then $Q_\beta$ is again a polynomial in
$S_1,\dots,S_q$ with $\theta$-dependent coefficients: differentiating $Q_\alpha$ in $\theta_j$ affects only its
coefficients, not its degree in $S$, so $\deg(\partial_{\theta_j}Q_\alpha)\le\deg Q_\alpha\le p\,m$, while
\begin{equation*}
    \deg(Q_\alpha P_j)=\deg Q_\alpha+\deg P_j\le p\,m+p=p\,(m+1),
\end{equation*}
whence $\deg Q_\beta\le p\,(m+1)=p\,|\beta|$. Moreover the coefficients of $Q_\beta$ are sums and products of
the infinitely continuously differentiable coefficients of $Q_\alpha$ and $P_j$ and of their $\theta$-derivatives, hence themselves
infinitely continuously differentiable. This completes the induction and establishes \eqref{eq:app_gen_1}.

\vspace{0.5 em}
\noindent
\emph{Step 2.}
Fix a maximal differential order $M\ge1$. The number of parameter derivatives $\partial^\alpha_\theta$ with
$|\alpha|\le M$ on a $d_2$-dimensional parameter space equals the number of multi-indices
$\alpha\in\mathbb N_0^{d_2}$ with $|\alpha|\le M$, namely
\begin{equation*}
    N_{\mathrm{der}}(M):=\binom{M+d_2}{d_2}=\frac{M^{d_2}}{d_2!}+\mathcal{O}\big(M^{d_2-1}\big)
    \qquad (M\to\infty).
\end{equation*}
By \eqref{eq:app_gen_1}, each such derivative, divided by $f$, is a polynomial in $S_1,\dots,S_q$ of total
degree at most $pM$; any linear combination $\sum_{|\alpha|\le M}c_\alpha(\theta)\partial^\alpha_\theta f$
therefore equals $f$ times a polynomial lying in the span of the monomials $S_1^{\gamma_1}\cdots S_q^{\gamma_q}$
with $\gamma_1+\cdots+\gamma_q\le pM$. The number of such monomials is at most
\begin{equation*}
    N_{\mathrm{pol}}(M):=\binom{pM+q}{q}=\frac{p^{\,q}M^{q}}{q!}+\mathcal{O}\big(M^{q-1}\big)
    \qquad (M\to\infty).
\end{equation*}
Since $q<d_2$,
\begin{equation*}
    \frac{N_{\mathrm{der}}(M)}{N_{\mathrm{pol}}(M)}
    =\frac{q!}{p^{\,q}\,d_2!}\,M^{\,d_2-q}\big(1+o(1)\big)\xrightarrow{M\to\infty}\infty,
\end{equation*}
so there exists an integer $M^*\ge1$ with $N_{\mathrm{der}}(M^*)>N_{\mathrm{pol}}(M^*)$. We fix such an $M^*$
for the remainder of the proof.

\vspace{0.5 em}
\noindent
\emph{Step 3.}
We seek coefficient functions $c_\alpha(\theta)$, indexed by $\alpha$ with $|\alpha|\le M^*$, not all zero,
such that
\begin{equation*}
    \mathcal L_\theta f(x\mid\theta)
    =\sum_{|\alpha|\le M^*}c_\alpha(\theta)\,\partial^\alpha_\theta f(x\mid\theta)=0
\end{equation*}
for $\lambda$-almost every $x\in\mathcal X$ and all $\theta$ in some open $U\subset \Theta^\circ$. By \eqref{eq:app_gen_1} and the strict positivity of $f$, this identity is equivalent to
\begin{equation}
    \sum_{|\alpha|\le M^*}c_\alpha(\theta)\,Q_\alpha\big(S(x);\theta\big)=0, \quad \theta \in U.
    \label{eq:app_gen_2}
\end{equation}
Let $m_1,\dots,m_{N_r}$ enumerate the distinct monomials in $S_1,\dots,S_q$, of total degree at most $pM^*$,
that appear across the polynomials $\{Q_\alpha:|\alpha|\le M^*\}$, so that $N_r\le N_{\mathrm{pol}}(M^*)$;
isolating these monomials as a basis removes any dependence coming from the sufficient-statistic map
$S:\mathcal X\to\mathbb R^q$. Writing $N_c:=N_{\mathrm{der}}(M^*)$ for the number of unknown coefficients, for any fixed $\theta$, 
define the matrix $\Phi(\theta)\in\mathbb R^{N_r\times N_c}$ whose $(i,\alpha)$ entry is the coefficient of the
monomial $m_i$ in $Q_\alpha(\,\cdot\,;\theta)$. Collecting the coefficients into
$\mathbf c(\theta)=(c_{\alpha_1}(\theta),\dots,c_{\alpha_{N_c}}(\theta))^\top$, condition \eqref{eq:app_gen_2}
is equivalent to
\begin{equation}
    \Phi(\theta)\,\mathbf c(\theta)=0 , \quad \theta \in U.
    \label{eq:app_gen_3}
\end{equation}
Below we will prove that the above system holds on an open subset $U$ of the parameter space. We will use similar arguments as in the proof of Theorem~\ref{thm:homogeneous}.

First, notice that, by the choice of $M^*$, we have $N_c>N_r$, so \eqref{eq:app_gen_3} is an underdetermined homogeneous system;
in particular $\operatorname{rank}\Phi(\theta)\le N_r$ for every $\theta$.
Let $\rho:=\max_{\theta\in\Theta^\circ}\operatorname{rank}\Phi(\theta)\le N_r$
and set $R:=\{\theta\in\Theta^\circ:\operatorname{rank}\Phi(\theta)=\rho\}$. If $\rho=0$ then $\Phi\equiv0$ on
$\Theta^\circ$ and \eqref{eq:app_gen_3} is solved by any non-zero constant $\mathbf c$; assume therefore
$\rho\ge1$. As $R$ is a non-empty subset of $\Theta$, let $\theta_0$ be any point in $R$. Then, there exists a $\rho \times \rho$ submatrix $M(\theta_0)$ with $\det M(\theta_0) \neq 0$. After relabeling the monomials (rows) and the multi-indices (columns), we may assume that $M(\theta_0)$ is a $\rho$-principal submatrix \footnote{Given a $m\times n$ matrix $M$, a $k$-principal submatrix of $M$ is the matrix occupying the first $k$ columns and first $k$ rows of $M$.} of $\Phi(\theta_0)$. As the entries
of $\Phi$ are $C^\infty$, there exists an open set $U\subseteq \Theta^\circ $ containing $\theta_0$ such that $\det M(\theta)\ne0$ on $U$, where $M(\theta)$ is the $\rho$-principal submatrix of $\Phi(\theta)$.

Let $B(\theta)$ be the $\rho\times(N_c-\rho)$   submatrix of $\Phi(\theta)$ adjacent to $M(\theta)$, and $\mathbf{b}(\theta)$ be its first column.
Split
$\mathbf c=(\mathbf c_A,\mathbf c_B)$ conformably, with
$\mathbf c_A=(c_{\alpha_1},\dots,c_{\alpha_{\rho}})^\top$ and
$\mathbf c_B=(c_{\alpha_{\rho+1}},\dots,c_{\alpha_{N_c}})^\top$; since $N_c-\rho\ge N_c-N_r\ge1$, 
$\mathbf c_B$ has at least one entry. 
Fix $\mathbf c_B(\theta):=(1,0,\dots,0)^\top$ for every $\theta\in U$. The top $\rho$ rows of \eqref{eq:app_gen_3}
then read
\[
 M(\theta)\mathbf c_A(\theta)+B(\theta)\mathbf c_B(\theta)=0,
\]
or, equivalently, 
\[
 M(\theta)\mathbf c_A(\theta)=-\mathbf b(\theta).
\]
Since $\det M(\theta)\ne0$ on $U$, this determines $\mathbf c_A(\theta)$ uniquely given the above choice of
$\mathbf c_B(\theta)$, by Cramer's rule,
\begin{equation*}
    c_{\alpha_j}(\theta)=\frac{\det M_{\alpha_j}(\theta)}{\det M(\theta)},\qquad 1\le j\le r,
\end{equation*}
where $M_{\alpha_j}(\theta)$ is obtained from $M(\theta)$ by replacing its $j$-th column with
$-\mathbf b(\theta)$. This choice annihilates the top $\rho$ rows of $\Phi(\theta)\mathbf c(\theta)$ by
construction. For the remaining rows, note that on $U$, since $\operatorname{rank}\Phi(\theta)=\rho$ and
$M(\theta)$ is invertible, each of bottom $N_r - \rho$ rows 
of $\Phi(\theta)$
is the same linear combination of the $\rho$
top rows, using similar argument as in equality \eqref{eqn:prop_2_dependent_when_i_larger_than_r}, the bottom $N_r - \rho$ entries of $\Phi(\theta)\mathbf{c}(\theta)$ are   similarly annihilated. Hence $\Phi(\theta)\mathbf c(\theta)=0$ on $U$, so \eqref{eq:app_gen_2}, and with it
$\mathcal L_\theta f(x\mid\theta)=0$, holds for every $x\in\mathcal X$ and $\theta\in U$.

Finally, since the entries of $\Phi$ are $C^\infty$ on $U$ and $\det M$ is bounded away from zero there, each
$c_{\alpha_j}$ is $C^\infty$ on $U$, while the entries of $\mathbf c_B$ are constant; thus all coefficients of
$\mathcal L_\theta$ lie in $C^\infty(U)$. The coefficient set to $1$ has multi-index $\alpha_{r+1}$,
thus the operator $\mathcal L_\theta$ is non-trivial on $U$ in the sense of
Definition~\ref{def:operator}. This establishes the proposition.\subsection{Proof of Theorem~\ref{thm:shift}}
\label{app:proof:thm_shift}
\textit{Step 1: Construction of measure discrepancy.} 
The first step is to construct a non-vanishing everywhere function $h$ such that 
\begin{equation*}
    \int_{\Theta} g(\eta)f(x\mid\eta)d\eta =0 \text{ for all } x \in \mathcal{X}. 
\end{equation*}
Consider the linear difference-differential operator $\mathcal{L}_{\theta}$ defined as 
\begin{equation*}
    \mathcal{L}_{\theta}[f(x\mid\cdot)] = \sum_{k=1}^{K}\sum_{|\alpha| \leq r_k} c_{k,\alpha}(\theta)\partial^{\alpha}_{\theta}f(x\mid\theta+v_k). 
\end{equation*}
Then, for any function $u \in C^{\infty}_{c}(\Theta)$ such that $\text{supp}(u) \subset B$, we have 
\begin{equation*}
    \int_{\Theta} u(\theta) \left(\sum_{k=1}^K \sum_{|\alpha|\leq r_k} c_{k,\alpha}(\theta)\partial^{\alpha}_{\theta}f(x\mid\theta+v_k)\right)d\theta = 0,
\end{equation*}
and thus as our function $\text{supp}(u) \subset B$, and as the summations are finite and the integrands are continuous on a compact set, we can write the above equation as 
\begin{align}
\sum_{k=1}^K \sum_{|\alpha| \le r_k} \int_B \Big( c_{k, \alpha}(\theta) u(\theta) \Big) \partial_\theta^\alpha f(x\mid\theta + v_k) d\theta = 0. \label{eq:dung_shift_PDE_2}
\end{align}
As $u \in C^{\infty}_c(\Theta)$ and $\mathrm{supp}(u) \subset B$, we have 
\begin{align*}
\int_B c_{k,\alpha}(\theta)u(\theta) \big[ \partial_\theta^\alpha f(x \mid \theta + v_k) \big] \, d\theta = (-1)^{|\alpha|} \int_B f(x\mid \theta + v_k) \big[ \partial_\theta^\alpha ( c_{k,\alpha}(\theta)u(\theta)) \big] \, d\theta.
\end{align*}
Plugging in this result into  \eqref{eq:dung_shift_PDE_2}, we have 
\begin{equation}
\label{eq:dung_shift_PDE_4}
    \sum_{k=1}^K \sum_{|\alpha| \le r_k} (-1)^{|\alpha|} \int_B f(x\mid \theta + v_k) \partial_\theta^\alpha \big[ c_{k, \alpha}(\theta) u(\theta) \big] \, d\theta = 0.
\end{equation}
By applying the change of variables $\eta = \theta + v_k$, we have
$$\int_B f(x \mid \theta + v_k) \partial_\theta^\alpha \big[ c_{k, \alpha}(\theta) u(\theta) \big] \, d\theta = \int_{B_k} f(x \mid \eta) \partial_\eta^\alpha \big[ c_{k, \alpha}(\eta - v_k) u(\eta - v_k) \big] \, d\eta$$
where $B_k := B + v_k$. Plugging this equation into equation~\eqref{eq:dung_shift_PDE_4} leads to
\begin{align}
\sum_{k=1}^K \sum_{|\alpha| \le r_k} (-1)^{|\alpha|} \int_{B_k} f(x\mid \eta) \partial_\eta^\alpha \big[ c_{k, \alpha}(\eta - v_k) u(\eta - v_k) \big] \, d\eta = 0. \label{eq:dung_shift_PDE_5}
\end{align}
By the properties of the bump function $u$, the multi-index partial derivatives $\partial_\eta^\alpha \big[ c_{k, \alpha}(\eta - v_k) u(\eta - v_k) \big]$ are 0 on the complement $\Theta \setminus B_k$. By understanding $u=0$ outside $\Theta$, we can rewrite equation~\eqref{eq:dung_shift_PDE_5} as
$$\int_\Theta \left( \sum_{k=1}^K \sum_{|\alpha| \le r_k} (-1)^{|\alpha|} \partial_\eta^\alpha \big[ c_{k, \alpha}(\eta - v_k) u(\eta - v_k) \big] \right) f(x\mid \eta) \, d\eta = 0.$$
We define a function $h: \Theta \to \mathbb{R}$ as 
$$h(\eta) := \sum_{k=1}^K \sum_{|\alpha| \le r_k} (-1)^{|\alpha|} \partial_\eta^\alpha \big[ c_{k, \alpha}(\eta - v_k) u(\eta - v_k) \big].$$

Next we will use the fact that the sets $B_j$'s are pairwise disjoint. Indeed, arguing by contradiction, if there exists a $\zeta\in B_i\cap B_j$ for some $i\ne j$, then $\|\zeta-\theta_0-v_i\|<\rho$ and $\|\zeta-\theta_0-v_j\|<\rho$,
so by using triangle inequality, $\|v_i-v_j\|<2\rho$, contradicting $2\rho<\min_{i\ne j}\|v_i-v_j\|$. 
Now, from the the disjointness of the sets $B_{k}$ for $1 \leq k \leq K$, we have
$$h(\eta) \Big|_{B_k} = h_k(\eta) = \sum_{|\alpha| \le r_k} (-1)^{|\alpha|} \partial_\eta^\alpha \big[ c_{k, \alpha}(\eta - v_k) u(\eta - v_k) \big],$$
which indicates that $h(\eta) = \sum_{k = 1}^{K} h_{k}(\eta)$. For any $1 \leq k \leq K$, we define linear partial differential operator $\mathcal{L}_{k}$ acting on a test function $u \in C_c^\infty(\Theta)$ as
$$(\mathcal{L}_{k}u)(\theta) = \sum_{|\alpha| \le r} c_{k,\alpha}(\theta) \partial_\theta^\alpha u(\theta).$$
Using the similar argument as the proof of Theorem~\ref{thm:homogeneous} in Section \ref{app:proof:thm_homogeneous}, the adjoint operator $\mathcal{L}_{k}^*$ takes the following form $$(\mathcal{L}_{k}^* u)(\theta) = \sum_{|\alpha| \le r} (-1)^{|\alpha|} \partial_\theta^\alpha \left( c_{k,\alpha}(\theta) u(\theta) \right).$$
Furthermore, $\mathcal{L}_{k}^* u \in C_c^\infty(\Theta)$ for any $u \in C_c^\infty(\Theta)$. Putting these equations together, we have
$$h_{k}(\eta) = \big(\mathcal{L}_k^* u\big)(\eta - v_k). $$
As the operator $\mathcal{L}_{\theta}$ is non-trivial, it indicates that the coefficient function $c_{k_*, \alpha}$ is non-zero for some multi-index $\alpha$ and $1 \leq j_{*} \leq J$. Therefore, the PDE operator $\mathcal{L}_{j_{*}}$ is non-trivial. Using the argument similar to that of Section~\ref{app:proof:thm_homogeneous}, there exists a test function $u \in C_c^\infty(B)$ such that $\mathcal{L}_{j_*}^* u \not \equiv 0$. Since $g(\eta) \Big|_{B_{k_{*}}} = g_{k_{*}}(\eta) = \big(\mathcal{L}_{k_{*}}^* u\big)(\eta - v_{k_{*}})$, it indicates that $g \not \equiv 0$. 

\textit{Step 2 - Construct probability measure $G$.}
By employing the same argument using Fubini's theorem as Step 3, Section \ref{app:proof:thm_homogeneous}, noting that $\int_{\Theta}h(\eta)f(x\mid\eta) = 0$, we have  
\begin{equation*}
    \int_{\Theta} h(\eta)d\eta = 0.
\end{equation*}
Let $\Omega_h := \text{supp}(h) \subset \bigcup_{k=1}^K B_k \subset \Theta$. Recall from the assumption that
$$ g_*(\eta) \ge m > 0 \quad \text{for all } \eta \in \Omega_g. $$
Because $h \in C_c^\infty(\Theta)$, we define $M = \|h\|_{L^\infty(\Theta)}$, which is finite. Since $h \not \equiv 0$, it indicates that $M > 0$. Let
$\varepsilon := \frac{m}{2M} > 0$, then, it is clear that $\epsilon \|h\|_{L^\infty(\Theta)} = m/2$. Now, we define
$$g(\theta) := g_{*}(\theta) + \epsilon \cdot h(\theta).$$
From here, with the similar argument as that Step 4, Section~\ref{app:proof:thm_homogeneous},
we have $g \in L^1(\Theta)$ and non-negative almost everywhere and integrates to $1$. Hence, it is a density function of a continuous probability measure $G \in \mathcal{M}(\Theta)$, and satisfies $p_G = p_{G_*}$ for $\lambda$-almost everywhere. In addition, as $$d_{\text{TV}}(G, G_*) = \frac{1}{2}\int_{\Theta}|g(\theta)-g_*(\theta)|\,d\theta = \frac{1}{2}\int_{\Theta}|h(\theta)|\,d\theta > 0$$
we achieve that $G \neq G_{*}$, and thus, $W_1(G,G_*) \geq c> 0$.

\subsection{Proof of Proposition~\ref{pro:statistical_estimation}}
\label{app:proof:prop_estimation}
Let $G_1\ne G_2$ be two measures in $\mathcal P(\Theta)$ satisfy $p_{G_1}=p_{G_2}$ for $\lambda$-almost everywhere, and set
$c:=W_1(G_1,G_2)>0$. Since finitely supported measures are $W_1$-dense in $\mathcal P(\Theta)$ when
$\Theta$ is compact, we can construct a pair of discrete sequences, $\{G_1^{(m)}\}_{m=1}^\infty \subset \mathcal{D}(\Theta)$ and $\{G_2^{(m)}\}_{m=1}^\infty \subset D(\Theta)$, such that 
$W_1(G_i^{(m)},G_i)\to 0$ for $1 \leq i \leq 2$. Thus, for there exists a positive integer $M_1$ such that for all $m\geq M_1$, we have $W_1(G_i^{(m)},G_i)< c/4$ for $m \geq M$.
An application of the triangle inequality yields
\begin{equation*}
    W_1(G_1^{(m)},G_2^{(m)})\geq W_1(G_1,G_2) -  W_1(G_1^{(m)},G_1) -W_1(G_2^{(m)},G) > c/2.
\end{equation*}

As the mapping $\theta\mapsto f(x\mid \theta)$ is continuous on the compact set $\Theta$ for $\lambda$-almost every $x$, it is also bounded for the $x$ values. Therefore, we appeal to Lemma~\ref{lem:mixture_continuity} and infer that $\|p_{G_i^{(m)}}-p_{G_i}\|_1\to0$. Using the fact that $p_{G_1}=p_{G_2}$ almost everywhere, we obtain that
\begin{equation*}
    \|p_{G_1^{(m)}}-p_{G_2^{(m)}}\|_1
    \le\|p_{G_1^{(m)}}-p_{G_1}\|_1+ \|p_{G_1} - p_{G_2}\|_1 + \|p_{G_2}-p_{G_2^{(m)}}\|_1\xrightarrow{m\to\infty}0 .
\end{equation*}
Thus, there exists a positive integer $M_2$ such that $\|p_{G_1^{(m)}}-p_{G_2^{(m)}}\|_1\le 1/(2n)$ for all $m \geq M_2$. By sub-additivity of the total variation distance over product
measures, we have
$$\|P^n_{G_1^{(m)}}-P^n_{G_2^{(m)}}\|_1\le n\|p_{G_1^{(m)}}-p_{G_2^{(m)}}\|_1\le 1/2.$$
Let $M = \max\{M_1,M_2\}$, Le Cam's
two-point inequality \citep[e.g.,][Lemma 15.9]{wainwright2019high} then gives, for any estimator $\hat G_n:\mathcal X^n\to\mathcal D(\Theta)$ and $m \geq M$, 
\begin{equation*}
    \sup_{G_*\in\mathcal D(\Theta)}\mathbb E_{X_{1:n}\sim P_{G_*}^n}\big[W_1(\widehat G_n,G_*)\big]
    \geq    \frac{1}{2}\,\frac{W_1\big(G_1^{(m)},G_2^{(m)}\big)}{2}\left(1-\frac12\|P^n_{G_1^{(m)}}-P^n_{G_2^{(m)}}\|_1\right)
    \geq \frac{3c}{32}.
\end{equation*}
As the bound is uniform in $n$, we arrive at
$$\inf_{n\in\mathbb N}\inf_{\hat G_n} \sup_{G_*\in\mathcal D(\Theta)}\mathbb E_{X_{1:n}\sim P_{G_*}^n}\big[W_1(\widehat G_n,G_*)\big] \ge\frac{3c}{32}.$$
As a consequence, we reach the conclusion of the proposition. 

\subsection{Proof of Proposition~\ref{prop:translation}}
\label{app:proof:prop_translation}

Let $G\in\mathcal M(\Theta)$ satisfy $\mathcal AG=0$ for $\lambda$-almost everywhere. Since $h>0$, we have that
\begin{equation*}
    \int_\Theta\gamma(\theta)K\big(\tau(x)-\psi(\theta)\big)\,G(d\theta)=0 \text{ for almost all } x.
\end{equation*}
Since the function $\gamma$ is assumed to be continuous
on the compact set $\Theta$, it is also bounded. Thus, $\tilde G(d\theta):=\gamma(\theta)\,G(d\theta)$ defines a finite signed measure. Let
$\mu:=\tilde G\circ \psi^{-1}$ the push-forward measure of $\mu:=\tilde G$   by $\psi$, by construction  supported on the compact set $\psi(\Theta)\subset\mathbb R^{d_1}$. Because $\tau$
maps $\mathcal X$ onto $\mathbb R^{d_1}$, the above identity can be rewritten as
\begin{equation*}
    \int_{\psi(\Theta)}K(y-s)\,\mu(ds)=0, \text{ for almost all } y\in\mathbb R^{d_1}.
\end{equation*}
Taking Fourier transforms and using the properties of convolutions, we obtain \begin{equation*}
    \mathcal F[K](\xi) \,\mathcal F[\mu](\xi) =0 \text{ for all } \xi \in \mathbb R^{d_1}
\end{equation*}
Because $\mathcal F[\mu]$ is
continuous and $\mathcal F[K]\ne0$ almost everywhere by hypothesis, we deduce that $\mathcal F[\mu]\equiv0$ and hence
$\mu=0$. Since $\psi$ is a homeomorphism onto its image, this implies $\tilde G=0$, and since $\gamma>0$, we conclude that $G=0$. This proves injectivity of $\mathcal A$, as claimed.

\subsection{Proof of Proposition~\ref{prop:exp_family_identif}}
\label{app:proof:prop_exp_family_identif}

Let $G\in\mathcal M(\Theta)$ satisfy $\mathcal AG=0$ for $\lambda$-almost everywhere. As $h>0$,
\begin{equation*}
    \int_\Theta\exp\big(\langle\eta(\theta),T(x)\rangle\big)e^{-A(\theta)}\,G(d\theta)=0,
    \qquad\lambda\text{-almost every\ }x .
\end{equation*}
Because $\Theta$ is compact and $A$ continuous, $e^{-A}$ is bounded away from $0$ and $\infty$, so
$\tilde G(d\theta):=e^{-A(\theta)}\,G(d\theta)$ is a finite signed measure with
$\tilde G(S)=0\Leftrightarrow G(S)=0$ for any measurable set $S\subseteq\Theta$. The map
$\Psi:=\exp\circ\,\eta:\Theta\to(0,\infty)^q$, applied coordinate-wise, is continuous and injective on the
compact set $\Theta$, hence a homeomorphism onto its image $\Omega:=\Psi(\Theta)=\exp\{\eta(\Theta)\}$, which
is therefore compact in $(0,\infty)^q$. Let $\mu:=\tilde G\circ \Psi^{-1}$, so that $\mu$ is a finite
signed measure on $\Omega$, and since $\Psi$ is a bijection between $\Theta$ and $\Omega$,
$\mu(S)=0\Leftrightarrow\tilde G(\Psi^{-1}(S))=0\Leftrightarrow G(\Psi^{-1}(S))=0$ for any $S\subseteq \Omega$ measurable. That is $\mu=0$ if and only if $\tilde G=0$.

Fix $x\in\mathcal X$ and set $s:=T(x)\in\mathbb R^q$.
For $\theta\in\Theta$, writing $y=\Psi(\theta)=(e^{\eta_1(\theta)},\dots,e^{\eta_q(\theta)})$, we have
\begin{equation*}
    \exp\big(\langle\eta(\theta),T(x)\rangle\big)
    =\exp\Big(\sum_{i=1}^q s_i\,\eta_i(\theta)\Big)
    =\prod_{i=1}^q\big(e^{\eta_i(\theta)}\big)^{s_i}
    =\prod_{i=1}^q y_i^{\,s_i}=y^{\,s}.
\end{equation*}
Therefore, changing variables in the last display equation, we obtain
\begin{equation}
    \int_\Omega y^{\,s}\,\mu(dy)=0,\qquad s\in T(\mathcal X).
    \label{eq:app_exp_moment}
\end{equation}
Notice that this holds for every $s\in T(\mathcal X)$ because the original
identity holds for $\lambda$-almost every $x$ and $T$ is measurable. Note that each monomial $y\mapsto y^{\,s}$
is continuous and bounded on the compact set $\Omega\subset(0,\infty)^q$, so the integrals in
\eqref{eq:app_exp_moment} are well defined and finite.

Since $T(\mathcal X)$ is determining for $\Omega$, the linear
span of $\{y\mapsto y^{\,s}:s\in T(\mathcal X)\}$ is dense in $C(\Omega)$ under the supremum norm. Given
$w\in C(\Omega)$ and $\varepsilon>0$, choose a finite linear combination
$w_\varepsilon=\sum_k a_k\,y^{\,s_k}$ with $s_k\in T(\mathcal X)$ and $\|w-w_\varepsilon\|_\infty<\varepsilon$.
By \eqref{eq:app_exp_moment}, $\int_\Omega w_\varepsilon\,d\mu=\sum_k a_k\int_\Omega y^{\,s_k}\,d\mu=0$, so
\begin{equation*}
    \Big|\int_\Omega w\,d\mu\Big|
    =\Big|\int_\Omega (w-w_\varepsilon)\,d\mu\Big|
    \le\|w-w_\varepsilon\|_\infty\,\|\mu\|_{\mathrm{TV}}
    <\varepsilon\,\|\mu\|_{\mathrm{TV}} .
\end{equation*}
Letting $\varepsilon\to0$ gives $\int_\Omega w\,d\mu=0$ for every $w\in C(\Omega)$. By the Riesz
representation theorem, the only finite signed Radon measure on the compact set $\Omega$ that integrates every
continuous function to zero is the zero measure, so $\mu=0$, and therefore $G=0$, proving the injectivity of $\mathcal A$.

\subsection{Proof of Proposition~\ref{prop:super_exponential}}
\label{app:proof:prop_super_exponential}

Recall the kernel $f(x\mid\theta)=h(x)\exp\{\theta_1x+\theta_2e^{x^2}-A(\theta)\}$ with base measure
$h(x)=\exp(-2be^{x^2})$ and parameter space $\Theta=[-a,a]\times[-b,b]$.

\medskip
\emph{(a) Non-existence of a non-trivial parameter PDE.}
Suppose, for contradiction, that some non-trivial parameter differential operator
$\mathcal L_\theta=\sum_{|\alpha|\le r}c_\alpha(\theta)\partial^\alpha_\theta$ with continuous coefficients on
$\Theta^\circ$ annihilates $f$; that is, there is a multi-index $\alpha^*$ with $|\alpha^*|=r$ and
$c_{\alpha^*}\not\equiv0$. Write
\begin{equation*}
    f(x\mid\theta)=h(x)\,e^{-A(\theta)}\,E(x\mid\theta),\qquad E(x\mid\theta):=\exp\big(\theta_1x+\theta_2e^{x^2}\big).
\end{equation*}
Since $h(x)>0$ is independent of $\theta$, the identity $\mathcal L_\theta f=0$ is equivalent to
$\mathcal L_\theta[e^{-A(\theta)}E(x\mid\theta)]=0$. Expanding the derivatives of the product
$e^{-A(\theta)}E(x\mid\theta)$ by the Leibniz rule yields an operator $\tilde{\mathcal L}_\theta$ acting on
$E$, i.e.
\begin{equation}
    \tilde{\mathcal L}_\theta E(x\mid\theta)
    :=e^{A(\theta)}\,\mathcal L_\theta\big[e^{-A(\theta)}E(x\mid\theta)\big]
    =\sum_{|\beta|\le r}\tilde c_\beta(\theta)\,\partial^\beta_\theta E(x\mid\theta)=0.
    \label{eq:app_super_a1}
\end{equation}
Note that the coefficients $\tilde c_\beta(\theta)$ are linear combinations of products of $c_\alpha(\theta)$ and higher order partial
derivatives of $A(\theta)$. For the top-order terms $|\beta|=r$, the Leibniz rule gives
$\tilde c_\beta(\theta)=c_\beta(\theta)$, so in particular $\tilde c_{\alpha^*}(\theta)=c_{\alpha^*}(\theta)
\not\equiv0$; moreover $\mathcal L_\theta\equiv0$ if and only if $\tilde{\mathcal L}_\theta\equiv0$, so
$\tilde{\mathcal L}_\theta$ is non-trivial whenever $\mathcal L_\theta$ is.

Since $E(x\mid\theta)=\exp(\theta_1x+\theta_2e^{x^2})$, differentiation gives
\begin{equation*}
    \partial^\beta_\theta E(x\mid\theta)
    =\frac{\partial^{|\beta|}}{\partial\theta_1^{\beta_1}\partial\theta_2^{\beta_2}}
    \exp\big(\theta_1x+\theta_2e^{x^2}\big)
    =x^{\beta_1}\big(e^{x^2}\big)^{\beta_2}E(x\mid\theta).
\end{equation*}
Substituting into \eqref{eq:app_super_a1} and using $E(x\mid\theta)>0$ for all $x\in\mathbb R$,
\begin{equation}
    \sum_{|\beta|\le r}\tilde c_\beta(\theta)\,x^{\beta_1}e^{\beta_2x^2}=0,
    \qquad x\in\mathbb R .
    \label{eq:app_super_a2}
\end{equation}
We now show that the above relationship forces $\tilde c_\beta(\theta)=0$ for every $|\beta|\le r$, contradicting
non-triviality. Indeed, for a non-negative integer $k \leq r$, grouping the terms in \eqref{eq:app_super_a2} by the value $k$  of the second index $\beta_2$ yields the polynomial in $x$ $P_k(x,\theta):=\sum_{\beta_1}\tilde c_{(\beta_1,k)}(\theta)\,x^{\beta_1}$. Thus, we can 
re-rewrite \eqref{eq:app_super_a2} as
\begin{equation}
    \sum_{k=0}^{M}P_k(x,\theta)\,e^{kx^2}=0,\qquad x\in\mathbb R,
    \label{eq:app_super_a3}
\end{equation}
where $M\le r$ is the largest index with $P_M(\cdot,\theta)\not\equiv0$ (such an $M$ exists precisely because
$\tilde{\mathcal L}_\theta$ is non-trivial). 
We will study the asymptotic behavior of the last identity as $x \to \infty$. Dividing both sides of \eqref{eq:app_super_a3} by $e^{Mx^2}$, we obtain
\begin{equation}
    P_M(x,\theta)+\sum_{k=0}^{M-1}P_k(x,\theta)\,e^{(k-M)x^2}=0 .
    \label{eq:app_super_a4}
\end{equation}
For each $k<M$ one has $k-M<0$, so $P_k(x,\theta)\,e^{(k-M)x^2}\to0$ as $x\to\infty$; letting $x\to\infty$ in
\eqref{eq:app_super_a4} therefore gives $\lim_{x\to\infty}P_M(x,\theta)=0$. But a polynomial in $x$ either is
identically zero or diverges to $\pm\infty$ or tends to a non-zero constant, so this limit is impossible
unless $P_M(\cdot,\theta)\equiv0$, contradicting the choice of $M$. Hence $P_k(\cdot,\theta)\equiv0$ for every
$k$, so $\tilde c_\beta(\theta)\equiv0$ for all $|\beta|\le r$. Since $\tilde c_\beta=c_\beta$ at the top order,
the highest-order coefficients of $\mathcal L_\theta$ vanish. By descending induction on $|\alpha|$, we calculate that 
$c_\alpha\equiv0$ for all $|\alpha|\le r$. Equivalently, $\mathcal L_\theta$ is trivial, a contradiction. Thus
$f$ satisfies no non-trivial homogeneous parameter PDE; the first claim of the Proposition is proved.

\medskip
\emph{(b) Identifiability.}
Suppose there exists a  $\mu\in\mathcal M(\Theta)$ satisfying
\begin{equation}
    \int_\Theta f(x\mid\theta)\,\mu(d\theta)
    =\int_\Theta h(x)\exp\big(\theta_1x+\theta_2e^{x^2}-A(\theta)\big)\,\mu(d\theta)=0,
    \qquad x\in\mathbb R .
    \label{eq:app_super_b0}
\end{equation}
We will show that this relationship implies that $\mu=0$ (thus proving identifiability) in four steps.

\emph{Step 1: Reduction and holomorphic extension.}
As $h(x)=\exp(-2be^{x^2})>0$, \eqref{eq:app_super_b0} is equivalent to
\begin{equation*}
\int_\Theta\exp(\theta_1x+\theta_2e^{x^2}-A(\theta))\,\mu(d\theta)=0, \text{ for all } x \in \mathbb{R}. 
\end{equation*}
Define the signed measure
$\tilde\mu$ by $\tilde\mu(d\theta):=e^{-A(\theta)}\,\mu(d\theta)$; since $\Theta$ is compact and $A$ is
continuous, $e^{-A}$ is bounded away from $0$ and $\infty$, so $\tilde\mu$ is finite and $\tilde\mu=0$ if and only if $\mu=0$. Thus, identity \eqref{eq:app_super_b0} becomes
\begin{equation}
    \int_\Theta\exp\big(\theta_1x+\theta_2e^{x^2}\big)\,\tilde\mu(d\theta)=0,
    \qquad x\in\mathbb R.
    \label{eq:app_super_b1}
\end{equation}
Let $F:\mathbb C^2\to\mathbb C$ be the Laplace transform of $\tilde\mu$, i.e. 
\begin{equation*}
  (w_1,w_2)\in \mathbb{C}^2 \mapsto  F(w_1,w_2):=\int_\Theta\exp\big(w_1\theta_1+w_2\theta_2\big)\,\tilde\mu(d\theta).
\end{equation*}
Because $\tilde\mu$ is finite with support in the compact box $[-a,a]\times[-b,b]$, the Paley--Wiener theorem
implies that $F$ is entire on $\mathbb C^2$ and satisfies
\begin{equation*}
    |F(w_1,w_2)|\le\|\tilde\mu\|_{\mathrm{TV}}\exp\big(a|\mathrm{Re}\,w_1|+b|\mathrm{Re}\,w_2|\big),
\end{equation*}
so $F$ is of exponential type (see Definition~\ref{def:exp_type} below).

\emph{Step 2: The diagonal slice and its Taylor expansion.}
Consider the entire function $M(z):=F(z,e^{z^2})$, a composition of entire functions. By
\eqref{eq:app_super_b1}, $M(x)=F(x,e^{x^2})=0$ for every $x\in\mathbb R$.
The real line is a subset of $\mathbb C$ with an accumulation point lying in $\mathbb C$,
so the identity theorem gives $M(z)=F(z,e^{z^2})=0$ for all $z\in\mathbb C$. Since $F$ is entire on
$\mathbb C^2$, it has a globally convergent Taylor expansion in its second argument,
\begin{equation*}
    F(w_1,w_2)=\sum_{n=0}^\infty c_n(w_1)\,w_2^{\,n},
    \qquad
    c_n(w_1)=\frac1{n!}\int_\Theta\theta_2^{\,n}\exp(w_1\theta_1)\,\tilde\mu(d\theta).
\end{equation*}

\emph{Step 3: Fast decay of the leading coefficient.}
By the Paley--Wiener theorem each $c_n$ is entire of exponential type. Evaluating the identity $M(z)=0$
through the expansion gives $\sum_{n\ge0}c_n(z)e^{nz^2}=0$ for all $z \in \mathbb{C}$, and substituting $z=iy$ with
$y\in\mathbb R$ yields that
\begin{equation}
    \sum_{n=0}^\infty c_n(iy)\,e^{-ny^2}=0
    \quad \text{or, eqiuivalently, } \quad
    c_0(iy)=-\sum_{n=1}^\infty c_n(iy)\,e^{-ny^2}.
    \label{eq:app_super_b3}
\end{equation}
From the integral formula for $c_n$, and using $|\tilde\mu|\le\|\tilde\mu\|_{\mathrm{TV}}$ and
$|\theta_2|\le b$,
\begin{equation*}
    |c_n(iy)|\le\frac1{n!}\int_\Theta|\theta_2|^{\,n}\,\big|e^{iy\theta_1}\big|\,|\tilde\mu|(d\theta)
    \le\frac{b^{\,n}}{n!}\,\|\tilde\mu\|_{\mathrm{TV}} .
\end{equation*}
Using these bounds and \eqref{eq:app_super_b3},
\begin{equation*}
    |c_0(iy)|\le\sum_{n=1}^\infty|c_n(iy)|\,e^{-ny^2}\le\|\tilde\mu\|_{\mathrm{TV}}\sum_{n=1}^\infty\frac{b^{\,n}}{n!}e^{-ny^2}=\|\tilde\mu\|_{\mathrm{TV}}\,e^{-y^2}\sum_{n=1}^\infty\frac{b^{\,n}}{n!}e^{-(n-1)y^2}.
\end{equation*}
Noting that $e^{-(n-1)y^2} \leq 1$, we have 
\begin{equation*}
    \sum_{n=1}^\infty\frac{b^{\,n}}{n!}e^{-(n-1)y^2} \leq \sum_{n=1}^{\infty} \frac{b^{\,n}}{n!} = e^b-1. 
\end{equation*}
As a result, with $C:=\|\tilde\mu\|_{\mathrm{TV}}(e^{b}-1)$ we obtain 
\[
|c_0(iy)|\le Ce^{-y^2},
\]
for all $y\in\mathbb R$.

\emph{Step 4: Logarithmic bound forces $c_0\equiv0$, then induction.}
Set $g(z):=c_0(iz)$, an entire function of exponential type with $|g(y)|\le Ce^{-y^2}$ on $\mathbb R$. Then,
$\log|g(y)|\le\log C-y^2$, leading to the following estimation for the logarithmic integral
\begin{equation*}
    \int_{-\infty}^\infty\frac{\log|g(y)|}{1+y^2}\,dy
    \le\int_{-\infty}^\infty\frac{\log C-y^2}{1+y^2}\,dy
    =\log C\int_{-\infty}^\infty\frac{dy}{1+y^2}-\int_{-\infty}^\infty\frac{y^2}{1+y^2}\,dy=-\infty,
\end{equation*}
the last integral diverging because $y^2/(1+y^2)\to1$. By Lemma~\ref{lem:log_bound}, an entire function of
exponential type whose logarithmic integral diverges to $-\infty$ must be identically equal to 0, hence $c_0\equiv g\equiv0$.

We now argue by induction. Suppose $c_k\equiv0$ for all $0\le k<m$. Then \eqref{eq:app_super_b3} reduces to
$\sum_{n\ge m}c_n(iy)e^{-ny^2}=0$, so
\begin{equation*}
    c_m(iy)\,e^{-my^2}=-\sum_{n=m+1}^\infty c_n(iy)\,e^{-ny^2},
\end{equation*}
and the same estimate as for $c_0$ gives
\begin{equation*}
    |c_m(iy)|\le\|\tilde\mu\|_{\mathrm{TV}}\,e^{-y^2}\sum_{k=1}^\infty\frac{b^{\,m+k}}{(m+k)!}\le C_me^{-y^2}
\end{equation*}
for a constant $C_m$. As above, we have 
\begin{equation*}
    \int_{-\infty}^\infty\log|c_m(iy)|/(1+y^2)\,dy=-\infty,
\end{equation*}
so Lemma~\ref{lem:log_bound} forces $c_m\equiv0$. By induction $c_n\equiv0$ for every $n\ge0$.

\emph{Step 5: Conclusion.}
We have proved that, for all $n\ge0$ and $y\in\mathbb R$,
\begin{equation*}
    c_n(iy)=\frac1{n!}\int_\Theta\theta_2^{\,n}\exp(iy\theta_1)\,\tilde\mu(d\theta)=0 .
\end{equation*}
Fix $n$ and define the finite signed measure $\nu_n$ on $[-a,a]$ by
$\nu_n(E):=\int_{E\times[-b,b]}\theta_2^{\,n}\,\tilde\mu(d\theta)$ for any measurable $E\subseteq[-a,a]$. The display
states that $\mathcal F[\nu_n](y)=0$ for all $y$, so by uniqueness of the Fourier transform, $\nu_n=0$.
Consequently, for every continuous $\phi$ on $[-a,a]$,
$\int_\Theta\phi(\theta_1)\,\theta_2^{\,n}\,\tilde\mu(d\theta)=0$, and therefore
$\int_\Theta \theta_1^{\,k}\theta_2^{\,n}\,\tilde\mu(d\theta)=0$ for all $k,n\ge0$. Hence,
$\int_\Theta P(\theta)\,\tilde\mu(d\theta)=0$ for every polynomial $P$. Since polynomials are dense in
$C(\Theta)$ by the Stone--Weierstrass theorem, $\int_\Theta w\,\tilde\mu(d\theta)=0$ for all $w\in C(\Theta)$,
which forces $\tilde\mu=0$ and thus $\mu=0$. 

\section{Auxiliary results and proofs}
\label{app:auxiliary}

Below, we collect some auxiliary results used in the proofs.
\begin{definition}\label{def:exp_type}
An entire function $g:\mathbb C\to\mathbb C$ is of \emph{exponential type} if there exist $C>0$ and
$\tau\ge0$ with $|g(z)|\le Ce^{\tau|z|}$ for all $z\in\mathbb C$.
\end{definition}

The following logarithmic-integral bound is the analytic device behind
Proposition~\ref{prop:super_exponential}(b): an entire function of exponential type that is bounded on the
real axis cannot have its log-modulus decay at a quadratic rate there unless it vanishes identically. This result is commonly referred to as Cartwright's version of Levinson's theorem

\begin{lemma}[after {\citealp[Section~III.G]{koosis1988logarithmic}}]\label{lem:log_bound}
If $g$ is entire of exponential type, bounded on $\mathbb R$, and $g\not\equiv0$, then
\begin{equation*}
    \int_{-\infty}^{\infty}\frac{\log|g(y)|}{1+y^2}\,dy>-\infty.
\end{equation*}
\end{lemma}

\begin{proof}
\emph{Step 1: Reduction to an even function with a controlled germ at $0$.}
By Definition~\ref{def:exp_type} there are $C\ge1$, $\tau\ge0$ with $|g(z)|\le Ce^{\tau|z|}$, and $M\ge1$ with
$|g(x)|\le M$ on $\mathbb R$. Suppose that $g$ has a zero of order $k$ at the origin, we can write $g(z)=c\,z^k\eta(z)$ with
$\eta$ entire, $\eta(0)=1$, and $c \neq 0$. For $|z| \geq 1$, we have 
\begin{equation*}
    |\eta(z)| \leq \frac{|g(z)|}{|c||z|^k} \leq \frac{C}{|c|}e^{\tau|z|}.  
\end{equation*}
As $\eta$ is continuous, it is bounded on $|z|\leq 1$. Thus, $\eta$ retains the exponential type $\tau$. Similarly, $|\eta(x)| \le \frac{M}{|c|}$ for $|x| \ge 1$, and $\eta(x)$ is bounded on $[-1, 1]$. Thus, $\eta(x)$ is bounded on $\mathbb{R}$.

Now we construct $G(z):=\eta(z)\eta(-z)$, which is even, entire of exponential type $2\tau$, bounded on $\mathbb R$ by
some $B^2$, with $G(0)=1$, $G'(0)=0$, and $\log|G(z)|=O(|z|^2)$ near $0$. 
Then, as 
\begin{equation*}
    \log|g(z)| + \log|g(-z)| = 2\log(c) + 2k\log|z| + \log|G(z)|,
\end{equation*}
we achieve the following equality 
\begin{equation*}
    \int_{-\infty}^{\infty} \frac{\log|g(y)|}{1+y^2}\,dy = \int_{-\infty}^{\infty} \frac{\log|c| + k\log|y|}{1+y^2}\,dy + \frac{1}{2}\int_{-\infty}^{\infty}\frac{\log|G(y)|}{1+y^2}\,dy. 
\end{equation*}
Noting that function $\frac{\log |c| + k\log|y|}{1+y^2}$ is integrable in entire $\mathbb{R}$ for $c \neq 0$, our problem reduces to 
\begin{equation}
    \label{eqn:dung_levinson_theorem_reduced_form}
    \int_{-\infty}^{\infty}\frac{\log|G(y)|}{1+y^2}\,dy > -\infty
\end{equation}

\emph{Step 2: Local weighted logarithmic integrability.} We prove that for any positive $R$, we have 
\begin{equation}
\label{eqn:dung_lemma_1_near_zero_integrability}
    \int_{-R}^R |\log|G(x)||\left(\frac{1}{x^2}-\frac{1}{R^2}\right)\,dx < \infty
\end{equation}
Indeed, as $G$ is a non-zero entire function in $\mathbb{C}$, there are at most a finite number of zero of $G$ at $[-R,R]$, let us say $\{x_j\}_{1\leq j\leq k}$. Let $\delta > 0$ be a sufficient small real number such that the intervals $(x_j-\delta, x_j+\delta)$ and $(-\delta,\delta)$ are mutually disjoint. Then, the function $|G(x)|$ are positive in the compact set $\Omega := [-R,R] \setminus \bigcup_{j=1}^{k}(x_j-\delta, x_j+\delta)$, thus there exists a constant $C > 0$ such that $|\log|G(x)|| \leq C$, thus 
\begin{equation}
\label{eqn:dung_lemma_1_non_singularity_integral}
    \int_{\Omega} |\log|G(x)||\left(\frac{1}{x^2} +\frac{1}{R^2}\right)\,dx \leq C\int_{-R}^R\left(\frac{1}{\delta^2}+\frac{1}{R^2}\right)\,dx = \frac{2RC}{\delta^2} + \frac{2C}{R}. 
\end{equation}

Now we prove the integrability of $|\log|G(x)||\left(\frac{1}{x^2} - \frac{1}{R^2}\right)$ in the intervals $(-\delta,\delta)$ and each $(x_j-\delta,x_j+\delta)$. 
For interval $(-\delta,\delta)$, as $\log|G(x)| = O(|x|^2)$ and function $x^2(\frac{1}{x^2} + \frac{1}{R^2})$ is integrable in $(-\delta,\delta)$, we have 
\begin{equation}
\label{eqn:dung_lemma_1_zero_non_singularity_integral}
    \int_{-\delta}^{\delta}|\log|G(x)|| \left(\frac{1}{x^2} + \frac{1}{R^2}\right)\,dx < \infty. 
\end{equation}

For each interval $(x_j-\delta,x_j+\delta)$, suppose that $x_j$ a zero of $G$ with finite multiplicity $m_j \ge 1$. We can express $G(x) = (x-x_j)^{m_j} \kappa(x)$, where $\kappa \neq 0$ in $(x_j -\delta, x_j+\delta)$. Thus, we have
$$ \left| \frac{\log|G(x)|}{x^2} \right| \le \frac{m_j \big|\log|x-x_j|\big|}{x^2} + \frac{\big|\log|\kappa(x)|\big|}{x^2}.$$

Since $(x_j-\delta,x_j + \delta)$ and $(-\delta,\delta)$ are disjoint, we have $|x| > \delta$ for each $x \in (x_j-\delta,x_j + \delta)$.  The term $\big|\log|\eta(x)|\big|$ is bounded because $\kappa(x_j) \neq 0$. For the logarithm term, we have 
$$ \int_{x_j-\delta}^{x_j+\delta} \big|\log|x-x_j|\big| dx = \int_{-\delta}^\delta |\log|t|| dt = 2\delta(1 - \log\delta) < \infty.$$
Thus, we achieve the estimation 
\begin{equation}
\label{eqn:dung_lemma_1_non_zero_singularity_integral}
    \int_{x_j-\delta}^{x_j+\delta} |\log|G(x)|| \left(\frac{1}{x^2} + \frac{1}{R^2}\right)\,dx < \infty. 
\end{equation}
By combining the estimation from \eqref{eqn:dung_lemma_1_non_singularity_integral}, \eqref{eqn:dung_lemma_1_zero_non_singularity_integral} and \eqref{eqn:dung_lemma_1_non_zero_singularity_integral}, we achieve the estimation 
\begin{equation*}
    \int_{-R}^R |\log|G(x)||\left(\frac{1}{x^2}+\frac{1}{R^2}\right)\,dx < \infty. 
\end{equation*}
In particular, this estimation implies \eqref{eqn:dung_lemma_1_near_zero_integrability}. When $R = 1$, we similarly achieve that  
\begin{equation}
\label{eqn:dung_lemma_1_log_G_over_x_square}
    \int_{-1}^1 \frac{|\log|G(x)||}{x^2}\,dx < \infty. 
\end{equation}

\emph{Step 3: Carleman's formula.} In the rest of this section, we denote $u(z)=\log|G(z)|$ and
\begin{equation*}
    v(z)=\mathrm{Im}\left(-\frac{1}{z}-\frac{z}{R}\right) = \sin\theta\left(\frac{1}{r} - \frac{r}{R^2}\right) \text{ for } z=re^{i\theta},\, r\geq 0,\, \theta \in [0,2\pi).
\end{equation*}
In this step we give a demonstration for the result known as Carleman formula, which is an important part of our proof (for a systematic reference, see \cite{levin1980distribution})
\begin{equation}
\label{eqn:dung_carleman_formula}
    \int_{-R}^R \log|G(x)| \left(\frac{1}{x^2} - \frac{1}{R^2}\right) dx + \frac{2}{R} \int_0^\pi \log|G(Re^{i\theta})| \sin\theta \, d\theta = 2\pi \sum_{z_k \in \mathcal{Z}_{+}} m_k v(z_k),
\end{equation}
where $\mathcal{Z}_{+}:= \{z_k \in \mathbb{C} : \text{Im}(z_k) > 0, |z_k| < R, G(z_k) = 0\}$.
Let $D_R^+$ be the open upper half-disk $\{z = r e^{i\theta} \in \mathbb{C} : r < R, \theta \in (0, \pi)\}$. Let $D_{R,\epsilon}^+$ be $D_R^+$ minus small disks of radius $\epsilon$ around each zero $z_k$ of $G$, and minus small half-disks of radius $\epsilon$ around the origin and any real roots $x_k$, namely,
$$ D_{R,\epsilon}^+ = \left\{ z \in \mathbb{C} : |z| < R, \, \text{Im}(z) > 0 \right\} \setminus \left( \overline{B_\epsilon(0)} \cup \bigcup_{x_k \in \mathcal{Z}_{\mathbb{R}}} \overline{B_\epsilon(x_k)} \cup \bigcup_{z_k \in \mathcal{Z}_+} \overline{B_\epsilon(z_k)} \right)$$
where $\mathcal{Z}_{\mathbb{R}} = \{x_k \in \mathbb{R} : |x_k| <R,\ G(x_k) = 0\}$ be the set of real zeros of $G$ with absolute value less than $R$. As $\eta$ is a non-zero function, $G$ is a non-zero holomorphic function too, which implies that the number of zero point of $G$ in ${D_R^{+}}$ is finite, in other words, $|\mathcal{Z}_{\mathbb{R}}|, |\mathcal{Z}_{+}| < \infty$. Here, we denote $\overline{B_\epsilon(x)} = \left\{ z \in \mathbb{C} : |z - x| \le \epsilon \right\}$ for any $x \in \mathbb{C}$ and $\epsilon > 0$.

\textit{Step 3.1: Green identity.} By construction, all points where $G(z) = 0$, namely, the zeros $\mathcal{Z}_{\mathbb{R}}$ and $\mathcal{Z}_+$ have been removed along with an open neighborhood of radius $\epsilon$ surrounding them. Therefore, $G(z) \neq 0$ for all $z \in D_{R,\epsilon}^+$. Since $G(z)$ is a holomorphic function, $u(z) = \log|G(z)|$ is a real harmonic function on open set $D_{R,\epsilon}^+$. Thus, $u \in C^2(D^{+}_{R,\epsilon})$ and whose Laplacian exactly 0 in ${D_{R,\epsilon}^+}$
$$\Delta u = \frac{\partial^2 u}{\partial x^2} + \frac{\partial^2 u}{\partial y^2} = 0.$$

Regarding the function $v$, since $-\frac{1}{z} - \frac{z}{R^2}$ is a complex analytic function everywhere except at its lone pole at the origin $z = 0$ and we remove the origin and its surrounding half-disk $\overline{B_\epsilon(0)}$, the point $z = 0$ is not in $\overline{D_{R,\epsilon}^+}$. It indicates that $v \in C^2(\overline{D_{R,\epsilon}^+})$ and satisfies the Laplace's equation
$$\Delta v = \frac{\partial^2 v}{\partial x^2} + \frac{\partial^2 v}{\partial y^2} = 0.$$

Using Green's second identity, noting that $\Delta u = \Delta v = 0$ in $D^{+}_{R,\epsilon}$, we have 
$$0 = \int_{D_{R,\epsilon}^+} (u \Delta v - v \Delta u) \, dA = -\oint_{\partial D_{R,\epsilon}^+} \left( u \frac{\partial v}{\partial n} - v \frac{\partial u}{\partial n} \right) ds.$$
Since $\Delta u = 0$ and $\Delta v = 0$ inside $D_{R,\epsilon}^+$, this equation becomes
$$ \oint_{\partial D_{R,\epsilon}^+} \left( u \frac{\partial v}{\partial n} - v \frac{\partial u}{\partial n} \right) ds = 0.$$
Let $\Xi(s) := u(s)\frac{\partial v}{\partial n}(s) - v(s) \frac{\partial u}{\partial n}(s)$, we now evaluate the integral of $\Xi$ over the $\partial D_{R,\epsilon}^+$ boundary as $\epsilon \to 0$. From the construction of $D_{R,\epsilon}^+$, we have for $\epsilon$ sufficiently small,
\begin{align}
\partial D_{R,\epsilon}^+ = \Gamma_{\text{Real}} \cup S_R \cup S_\epsilon \cup \left( \bigcup_{x_k \in \mathcal{Z}_{\mathbb{R}}} \gamma_{x_k} \right) \cup \left( \bigcup_{z_k \in \mathcal{Z}_{+}} \gamma_{z_k} \right) \label{eq:logarithmic_bound_1}
\end{align}
where $\Gamma_{\text{Real}}$ consists of the remaining flat line segments left behind on $[-R, R]$ after slicing out small intervals of width $2\epsilon$ around the origin and any real roots $x_k$; $\gamma_{x_k}$ are the tiny clockwise semicircles of radius $\epsilon$ looping over the real roots $x_k$ into the upper half-plane to keep them outside our domain; $S_R$ is the counter-clockwise circular arc running from $(R, 0)$ to $(-R, 0)$ in the upper half-plane; $S_\epsilon$ is a small clockwise semicircle of radius $\epsilon$ centered at $z = 0$; $\gamma_{z_k}$ is a boundary of the closed disk of radius $\epsilon$ around $z_{k}$.

\textit{Step 3.2: Integral of $\Xi$ over $\Gamma_{\text{Real}}$. } We first evaluate $\lim_{\epsilon \to 0^+} \int_{\Gamma_{\text{Real}}} \Xi(s)\,ds$. From the definition of $\Gamma_{\text{Real}}$, direct calculation leads to
$$\int_{\Gamma_{\text{Real}}} \left( u \frac{\partial v}{\partial n} - v \frac{\partial u}{\partial n} \right) ds = \int_{-R}^R \phi(x) \chi_{I_{R,\epsilon}}(x) \, dx$$
where $\phi(x) = \log|G(x)| \left( \frac{1}{x^2} - \frac{1}{R^2} \right)$ and $\chi_{I_{R,\epsilon}}(x)$ is the characteristic indicator function of the set $I_{R,\epsilon}$, defined as:$$ I_{R,\epsilon} = [-R, R] \setminus \left( (-\epsilon, \epsilon) \cup \bigcup_{x_k \in \mathcal{Z}_{\mathbb{R}}} (x_k - \epsilon, x_k + \epsilon) \right).$$
It is clear that $\lim_{\epsilon \to 0^+} \chi_{I_{R,\epsilon}}(x) = 1$ for almost all $x \in [-R, R]$. Therefore, we have
\begin{equation*}
    \lim_{\epsilon \to 0^+} \phi(x) \chi_{I_{R,\epsilon}}(x) = \phi(x)
\end{equation*}
for almost surely $x \in [-R, R]$. Using the fact that $|\phi(x)|$ is integrable in $[-R,R]$ in Step 2 (see \eqref{eqn:dung_lemma_1_near_zero_integrability}), and putting these results together, an application of the dominated convergence theorem leads to
\begin{align}
\lim_{\epsilon \to 0^+} \int_{\Gamma_{\text{Real}}} \left( u \frac{\partial v}{\partial n} - v \frac{\partial u}{\partial n} \right) ds & = \lim_{\epsilon \to 0^+} \int_{-R}^R\phi(x) \chi_{I_{R,\epsilon}}(x) \, dx = \int_{-R}^R \left( \lim_{\epsilon \to 0^+} \phi(x) \chi_{I_{R,\epsilon}}(x) \right) dx \nonumber \\
& = \int_{-R}^R \log|G(x)| \left( \frac{1}{x^2} - \frac{1}{R^2} \right) dx. \label{eq:logarithmic_bound_2}
\end{align}

\textit{Step 3.3: Integral of $\Xi$ over $S_R$. } We now move to evaluate $\lim_{\epsilon \to 0^+} \int_{S_{R}} \Xi(s)\,ds$. Along the semicircle $S_{R}$, we have $v = 0$. Therefore, 
$$\lim_{\epsilon \to 0^+} \int_{S_{R}} \left( u \frac{\partial v}{\partial n} - v \frac{\partial u}{\partial n} \right) ds = \lim_{\epsilon \to 0^+} \int_{S_{R}} u \frac{\partial v}{\partial n} ds.$$
The semicircle $S_R$ is a 1-dimensional manifold parameterized smoothly by the map 
$$\gamma_R: [0, \pi] \to \mathbb{C}, \qquad \gamma_R(\theta) = R e^{i\theta} = (R\cos\theta, R\sin\theta).$$
To convert the line integral $\int_{S_R} u \frac{\partial v}{\partial n} \, ds$ into a standard Riemann integral over $[0, \pi]$, we determine the differential arc length element $ds$. The tangent vector is $\gamma_R'(\theta) = (-R\sin\theta, R\cos\theta)$. Therefore, we have $ds = \|\gamma_R'(\theta)\| \, d\theta = \sqrt{(-R\sin\theta)^2 + (R\cos\theta)^2} \, d\theta = R \, d\theta$. Direct computation yields that
$$\int_{S_R} u(z) \frac{\partial v}{\partial n}(z) \, ds = \int_0^\pi u(\gamma_R(\theta)) \cdot \frac{\partial v}{\partial n}(\gamma_R(\theta)) \cdot \|\gamma_R'(\theta)\| \, d\theta.$$
Plugging in $u(Re^{i\theta}) = \log|G(Re^{i\theta})|$, $\frac{\partial v}{\partial n}(\gamma_{R}(\theta)) = \frac{2}{R^2}\sin\theta$, and $ds = R \, d\theta$, we obtain that
$$\int_{S_R} u \frac{\partial v}{\partial n} \, ds = \int_0^\pi \log|G(Re^{i\theta})| \left( \frac{2}{R^2}\sin\theta \right) R \, d\theta = \frac{2}{R} \int_0^\pi \log|G(Re^{i\theta})| \sin\theta \, d\theta.$$
Putting these results together, we arrive at
\begin{align}
\lim_{\epsilon \to 0^+} \int_{S_{R}} \left( u \frac{\partial v}{\partial n} - v \frac{\partial u}{\partial n} \right) ds = \frac{2}{R} \int_0^\pi \log|G(Re^{i\theta})| \sin\theta \, d\theta. \label{eq:logarithmic_bound_3}
\end{align}

\textit{Step 3.4: Integral of $\Xi$ over $S_{\epsilon}$. } We now move to evaluate $\lim_{\epsilon \to 0^+} \int_{S_{\epsilon}} \Xi(s)\,ds$. The semicircle $S_{\epsilon}$ can be parameterized by the map 
$$\gamma_0: [0, \pi] \to \mathbb{C}, \qquad \gamma_0(\theta) = \epsilon e^{i(\pi - \theta)} = -\epsilon\cos\theta + i\epsilon\sin\theta.$$
The tangent vector of this path is $\gamma_0'(\theta) = \epsilon\sin\theta + i\epsilon\cos\theta$, which leads to $ds = \|\gamma_0'(\theta)\| \, d\theta = \sqrt{(\epsilon\sin\theta)^2 + (\epsilon\cos\theta)^2} \, d\theta = \epsilon \, d\theta$. Direct calculation yields that
$$\left| \int_{S_{\epsilon}} \left( u \frac{\partial v}{\partial n} - v \frac{\partial u}{\partial n} \right) ds\right| = \left| \int_0^\pi \left( u(\gamma_0(\theta))\frac{\partial v}{\partial n}(\gamma_0(\theta)) - v(\gamma_0(\theta))\frac{\partial u}{\partial n}(\gamma_0(\theta)) \right) \epsilon \, d\theta \right|.$$
By construction, $G(z)$ is an entire function with $G(0) = 1$ and $G'(0) = 0$. Its Taylor series expansion centered at the origin is:
$$G(z) = 1 + z^2 \psi(z)$$
where $\psi(z) = \sum_{n=2}^\infty a_n z^{n-2}$ is an entire function. Because $\psi(z)$ is continuous on  $\overline{B_1(0)}$, there exists $M_1 \ge 0$ such that for all $|z| = \epsilon \le 1$ we have
$$|G(z) - 1| \le \epsilon^2 M_1.$$
Invoking the inequality $|\log(1+x)| \le 2|x|$ for all $|x| \le \frac{1}{2}$, we choose $\epsilon \le \min\left(1, \frac{1}{\sqrt{2M_1}}\right)$ to ensure $|G(z)-1| \le \frac{1}{2}$. This allows us to bound $u = \log|G(z)|$ as 
$$|u(\gamma_0(\theta))| = \big|\text{Re}(\log G(z))\big| \le \big|\log G(z)\big| \le 2|G(z) - 1| \le 2M_1 \epsilon^2.$$
Furthermore, we can express $G'(z) = z \phi(z)$ where $\phi(z)$ is bounded by some $M_2 \ge 0$ on $\overline{B_1(0)}$. Then, we obtain that $$\left|\frac{\partial u}{\partial n}(\gamma_0(\theta))\right| = \left|\text{Re}\left( \frac{G'(z)}{G(z)} e^{i\theta} \right)\right| \le \frac{|G'(z)|}{|G(z)|} \le \frac{\epsilon M_2}{1 - \epsilon^2 M_1} \le 2M_2 \epsilon \qquad \text{for } \epsilon^2 M_1 \le \frac{1}{2}.$$
Similarly, since $|\sin\theta| \le 1$, for any $\epsilon < R$ we have
$$|v(\gamma_0(\theta))| \le \frac{1}{\epsilon} + \frac{\epsilon}{R^2} \le \frac{2}{\epsilon}, \quad \left|\frac{\partial v}{\partial n}(\gamma_0(\theta))\right| \le \frac{1}{\epsilon^2} + \frac{1}{R^2} \le \frac{2}{\epsilon^2}.$$
As $\epsilon$ is sufficiently small, putting these results together leads to
$$\left| \int_{S_{\epsilon}} \left( u \frac{\partial v}{\partial n} - v \frac{\partial u}{\partial n} \right) ds\right|  \le \int_0^\pi \left[ (2M_1 \epsilon^2) \cdot \left(\frac{2}{\epsilon^2}\right) + \left(\frac{2}{\epsilon}\right) \cdot (2M_2 \epsilon) \right] \epsilon \, d\theta \le \pi (4M_1 + 4M_2) \epsilon.$$
Therefore, we obtain that
\begin{align}
\lim_{\epsilon \to 0^+} \int_{S_{\epsilon}} \left( u \frac{\partial v}{\partial n} - v \frac{\partial u}{\partial n} \right) ds = 0. \label{eq:logarithmic_bound_4}
\end{align}
\emph{Step 3.5: Integral of $\Xi$ over $\gamma_{z_k}$. } We now move to $\lim_{\epsilon \to 0^+} \int_{\gamma_{z_k}} \Xi(s) ds$ for any $z_k \in \mathcal{Z}_{+}$. Indeed, the path $\gamma_{z_k}$ can be parameterized by
$$\gamma_{z_k}(\theta) = z_k + \epsilon e^{i(2\pi - \theta)} = z_k + \epsilon e^{-i\theta}, \qquad \theta \in [0, 2\pi].$$
Therefore, the differential arc length element is $ds = \|\gamma_{z_k}'(\theta)\|d\theta = \epsilon d\theta$. Assume that $z_k \in \mathcal{Z}_+$ be a zero of $G(z)$ with multiplicity $m_k \ge 1$. Then, we can factor $G(z) = (z - z_k)^{m_k} \psi(z)$, where $\psi(z)$ is analytic and satisfies $\psi(z_k) \neq 0$. Now, we expand $u(\gamma_{z_k}(\theta))$ near $z_{k}$ and obtain that
$$u(\gamma_{z_k}(\theta)) = \log|(\epsilon e^{-i\theta})^{m_k} \psi(z_k + \epsilon e^{-i\theta})| = m_k \log\epsilon + \log|\psi(z_k + \epsilon e^{-i\theta})| = m_k \log\epsilon + \log|\psi(z_k)| + O(\epsilon).$$
Similarly, we have
$$\frac{\partial u}{\partial n}(\gamma_{z_{k}}(\theta)) = \frac{m_k}{\epsilon} + O(1), \quad v(\gamma_{z_k}(\theta)) = v(z_k) + O(\epsilon), \quad \frac{\partial v}{\partial n}(\gamma_{z_{k}}(\theta)) = O(1).$$
Therefore, we obtain that
\begin{align*}
\int_{\gamma_{z_k}} \left( u \frac{\partial v}{\partial n} - v \frac{\partial u}{\partial n} \right) ds & = \int_0^{2\pi} \left[ \Big(m_k\log\epsilon + O(1)\Big)O(1) - \Big(v(z_k) + O(\epsilon)\Big)\left(\frac{m_k}{\epsilon} + O(1)\right) \right] \epsilon \, d\theta \\
& = \int_0^{2\pi} \left[ O(\epsilon \log\epsilon) - m_k v(z_k) - O(\epsilon) \right] d\theta. \label{eq:logarithmic_bound_5}
\end{align*}
As a consequence, we arrive at
\begin{align}
\lim_{\epsilon \to 0^+} \int_{\gamma_{z_k}} \left( u \frac{\partial v}{\partial n} - v \frac{\partial u}{\partial n} \right) ds = \int_0^{2\pi} -m_k v(z_k) \, d\theta = -2\pi m_k v(z_k). 
\end{align}
\emph{Step 3.6: Integral of $\Xi$ over $\gamma_{x_k}$.} We finally evaluate $\lim_{\epsilon \to 0^+} \int_{\gamma_{x_k}} \Xi(s) ds$ for any $x_k \in \mathcal{Z}_{\mathbb{R}}$. Indeed, we can parameterize $\gamma_{x_{k}}$ as 
$$\gamma_{x_k}(\theta) = x_k + \epsilon e^{-i\theta} \quad \text{for} \ \theta \in [\pi, 2\pi].$$ 
Direct calculation yields that
\begin{align*}
\int_{\gamma_{x_k}} \left( u \frac{\partial v}{\partial n} - v \frac{\partial u}{\partial n} \right) ds & = \int_\pi^{2\pi} \left[ \Big(m_k\log\epsilon + O(1)\Big)O(1) - O(\epsilon)\left(\frac{m_k}{\epsilon} + O(1)\right) \right] \epsilon \, d\theta \\
& = \int_\pi^{2\pi} \left[ O(\epsilon \log\epsilon) - O(\epsilon) \right] d\theta.
\end{align*}
Therefore, we arrive at
\begin{align}
\lim_{\epsilon \to 0^+} \int_{\gamma_{x_k}} \left( u \frac{\partial v}{\partial n} - v \frac{\partial u}{\partial n} \right) ds = 0. \label{eq:logarithmic_bound_6}
\end{align}
\emph{Step 3.7: Putting everything together. } Putting the results from equations~\eqref{eq:logarithmic_bound_1}-\eqref{eq:logarithmic_bound_6} together, we obtain that
$$\int_{-R}^R \log|G(x)| \left(\frac{1}{x^2} - \frac{1}{R^2}\right) dx + \frac{2}{R} \int_0^\pi \log|G(Re^{i\theta})| \sin\theta \, d\theta = 2\pi \sum_{z_k \in \mathcal{Z}_{+}} m_k v(z_k).$$

\textit{Step 4: Integrability of $|\log|G(x)||/x^2$ at infinity.} We now prove that 
\begin{equation*}
    \int_{|x|\geq 1} \frac{|\log|G(x)||}{x^2}\,dx < \infty. 
\end{equation*}
Since $v(z) > 0$ for all $z \in D_R^+$, Carleman formula in equation \eqref{eqn:dung_carleman_formula} leads to
\begin{equation}
\label{eqn:dung_lemma_1_LHS_carleman_greater_0}
    \int_{-R}^R \log|G(x)| \left(\frac{1}{x^2} - \frac{1}{R^2}\right) dx + \frac{2}{R} \int_0^\pi \log|G(Re^{i\theta})| \sin\theta \, d\theta \geq 0.
\end{equation}
Furthermore, since $|G(z)| \le \bar{C} e^{2\tau R}$ on $S_R$, we obtain that
\begin{equation}
\label{eqn:dung_lemma_1_integral_over_semicircle_less_than}
    \frac{2}{R} \int_0^\pi \log|G(Re^{i\theta})| \sin\theta \, d\theta \le \frac{2}{R} \int_0^\pi (\log \bar{C} + 2\tau R) \sin\theta \, d\theta = \frac{4\log \bar{C}}{R} + 8\tau.
\end{equation}
Collecting these bounds in \eqref{eqn:dung_lemma_1_LHS_carleman_greater_0} and \eqref{eqn:dung_lemma_1_integral_over_semicircle_less_than} leads to
\begin{equation*}
    \int_{-R}^R \log|G(x)| \left(\frac{1}{x^2} - \frac{1}{R^2}\right) dx \ge -4\frac{|\log \bar{C}|}{R} - 8\tau.
\end{equation*}
Since $\log|G(x)| = \log^+|G(x)| - \log^-|G(x)|$ where $\log^- t := \min\{0, \log t\}$ for any $t > 0$, this bound can be written as 
\begin{equation}
    \label{eqn:dung_lemma_1_lower_bound_for_log_G_depend_on_R}
\int_{-R}^R \log^-|G(x)| \left(\frac{1}{x^2} - \frac{1}{R^2}\right) dx \le \int_{-R}^R \log^+|G(x)| \left(\frac{1}{x^2} - \frac{1}{R^2}\right) dx + 4\frac{|\log \bar{C}|}{R} + 8\tau.
\end{equation}
When $1 \le |x| \le R$, from the hypothesis that $|G|$ is bounded in $\mathbb{R}$, there exists positive constant $M'$ not depending on $R$ such that $\log^+|G(x)| \le M'$. As a result, we obtain
\begin{equation}
    \label{eqn:dung_lemma_1_upper_bound_for_log_G_plus_depend_on_R}
    \int_{1 \le |x| \le R} \log^+|G(x)| \left(\frac{1}{x^2} - \frac{1}{R^2}\right) dx \le \int_{1 \le |x| \le R} \frac{M'}{x^2} dx = 2M'\left(1 - \frac{1}{R}\right) \le 2M'.
\end{equation}
Combining these bounds in \eqref{eqn:dung_lemma_1_lower_bound_for_log_G_depend_on_R} and \eqref{eqn:dung_lemma_1_upper_bound_for_log_G_plus_depend_on_R}, we arrive at
\begin{align}
\int_{1 \le |x| \le R} \log^-|G(x)| \left(\frac{1}{x^2} - \frac{1}{R^2}\right) dx \le 2M' + \int_{-1}^1 \log|G(x)| \left(\frac{1}{x^2} - \frac{1}{R^2}\right) dx +  4|\log \bar{C}| + 8\tau. \label{eq:logarithmic_bound_7}
\end{align}

Using the estimation in \eqref{eqn:dung_lemma_1_log_G_over_x_square}, we have 
$$\int_{-1}^1 \frac{|\log|G(x)||}{x^2} dx = W < \infty.$$
Therefore, we obtain that when $R > 1$,
\begin{align}
\int_{-1}^1 \log|G(x)| \left(\frac{1}{x^2} - \frac{1}{R^2}\right) dx \le 2\int_{-1}^1 \frac{|\log|G(x)||}{x^2} dx = 2W. \label{eq:logarithmic_bound_8}
\end{align}
Combining the results from equation~\eqref{eq:logarithmic_bound_7} and~\eqref{eq:logarithmic_bound_8} leads to
$$ \int_{1 \le |x| \le R} \log^-|G(x)| \left(\frac{1}{x^2} - \frac{1}{R^2}\right) dx \le 2M' + 2W + 4|\log \bar{C}| + 8\tau.$$
By taking $R \to \infty$ and using Monotone Convergence Theorem, this inequality eventually indicates that
$\int_{|x| \ge 1} \frac{\log^-|G|}{x^2} dx < \infty$. As $\int_{|x| \ge 1} \frac{\log^+|G|}{x^2} dx \leq 2M'$, these two results together leads to
$$\int_{|x| \ge 1} \frac{|\log|G(x)||}{x^2} dx < \infty.$$


\textit{Step 5: Conclusion. } Since $\frac{1}{1+x^2} < \frac{1}{x^2}$ and $\frac{\log|G(x)|}{x^2}$ is in $L^1([-1, 1])$,
we obtain that  
$$\int_{|x| \ge 1} \frac{|\log|G(x)||}{1 + x^2} dx < \int_{|x| \ge 1} \frac{|\log|G(x)||}{x^2} dx < \infty,$$
$$\int_{|x| \leq 1} \frac{|\log|G(x)||}{1 + x^2} dx < \int_{|x| \leq 1} \frac{|\log|G(x)||}{x^2} dx < \infty.$$
Consequently, we arrive at
$$ \int_{-\infty}^\infty \frac{|\log|G(x)||}{1+x^2} dx < \infty,$$
which indicates that \eqref{eqn:dung_levinson_theorem_reduced_form} is correct. We obtain the conclusion of the lemma.
\end{proof}

The next lemma supplies the continuity used in Appendix~\ref{app:proof:prop_estimation}. 

\begin{lemma}\label{lem:mixture_continuity}
Suppose that, for $\lambda$-almost every $x$, the map $\theta\mapsto f(x\mid\theta)$ is bounded and continuous on
$\Theta$. If $G_k,G\in\mathcal P(\Theta)$ and $W_1(G_k,G)\to0$, then $\|p_{G_k}-p_G\|_1\to0$.
\end{lemma}

\begin{proof}
On the compact $\Theta$, $W_1$ metrises weak convergence, so $G_k\to G$ weakly. For $\lambda$-almost everywhere\ $x$,
$f(x\mid\cdot)$ is bounded and continuous, whence 
\begin{equation*}
    p_{G_k}(x)=\int_\Theta f(x\mid\theta)\,dG_k(\theta)\to
\int_\Theta f(x\mid\theta)\,dG(\theta)=p_G(x), \text{ a.e.}
\end{equation*}
Since each $p_{G_k}$ and $p_G$ is a probability density and $p_{G_k} \to p_G$ almost everywhere with respect to measure $\lambda$, Scheffé's lemma yields
$\|p_{G_k}-p_G\|_1\to0$. This completes the proof of Lemma \ref{lem:mixture_continuity}. 
\end{proof}

\bibliographystyle{apalike} 
\bibliography{references}

\end{document}